\documentclass[11pt,a4paper,reqno]{amsart}
\usepackage{%
    amssymb%
    ,graphicx%
    ,mathrsfs%
    ,eucal%
    ,xcolor%
    ,enumitem%
    ,esint}
\usepackage[%
    pdfauthor={P. D'Ancona}%
    ,colorlinks=true%
    ,linkcolor=magenta%
    ,citecolor=cyan%
    ,urlcolor=blue%
    ,hyperfootnotes=false%
]{hyperref}

\numberwithin{equation}{section}
\newtheorem{theorem}{Theorem}[section]
\newtheorem{corollary}[theorem]{Corollary}
\newtheorem{lemma}[theorem]{Lemma}

\newtheorem{proposition}[theorem]{Proposition}
\theoremstyle{remark}
\newtheorem{remark}{Remark}[section]

\theoremstyle{definition}

\usepackage{stmaryrd}
  {\baselineskip0pt\normalfont\footnotesize
  \setlength\abovedisplayskip{1pt}
  \setlength\abovedisplayshortskip{1pt}
  \setlength\belowdisplayskip{1pt}
  \setlength\belowdisplayshortskip{1pt}
  \vskip2pt\par\noindent$\llbracket$\ignorespaces}%
  {\ignorespaces$\rrbracket$\vskip2pt}

\DeclareMathOperator{\sgn}{sgn}

\newcommand{\bra}[1]{\langle #1 \rangle}
\newcommand{\one}[1]{\mathbf{1}_{#1}}

\title[Sommerfeld condition]
{A limiting absorption principle
for the Helmholtz equation with variable coefficients}
\date{\today}    


\begin{document}
\maketitle

\centerline{\scshape Federico Cacciafesta}
\medskip
{\footnotesize
 \centerline{Dipartimento di Matematica}
   \centerline{Universit\`a degli studi di Padova}
   \centerline{ Via Trieste, 63, 35131 Padova, Italia}
   \centerline{{\em email adress:} cacciafe@math.unipd.it}
   \centerline{(corresponding author)}
} 
\medskip

\centerline{\scshape Piero D'Ancona}
\medskip
{\footnotesize 
	\centerline {Dipartimento di matematica}
	\centerline {Sapienza Universit\`{a} di Roma}
	\centerline {Piazzale A.~Moro 2, 00185 Roma, Italy}
	\centerline{{\em email adress:} dancona@mat.uniroma1.it
	}
	\medskip
\centerline{\scshape Renato Luc\`{a}}
\medskip
{\footnotesize 
	\centerline {Departement Matematik und Informatik}
	\centerline {Universit\"at Basel}
	\centerline {Spiegelgasse 1, 4051 Basel}
	\centerline{{\em email adress:} renato.luca@unibas.ch
	}

\begin{abstract}
  We prove a limiting absorption principle for a
  generalized Helmholtz equation on an exterior domain
  with Dirichlet boundary conditions
  \begin{equation*}
    (L+\lambda)v=f,
    \qquad
    \lambda\in \mathbb{R}
  \end{equation*}
  under a Sommerfeld radiation condition at infinity.
  The operator $L$ is a second order elliptic operator
  with variable coefficients; the principal part is a small,
  long range perturbation of $-\Delta$, while lower order
  terms can be singular and large.

  The main tool is a sharp uniform resolvent
  estimate, which has independent applications to the
  problem of embedded eigenvalues and to 
  smoothing estimates for dispersive equations.
\end{abstract}


\section{Introduction}\label{sec:int}

The Helmholtz equation
\begin{equation}\label{eq:helm0}
  \Delta v+\kappa^{2}v=f(x),
  \qquad
  \kappa\in \mathbb{R}
\end{equation}
on an exterior domain $\Omega=\mathbb{R}^{n}\setminus \Sigma$,
is used to model the scattering
by a compact obstacle $\Sigma$ of waves
generated by a source $f(x)$.
The operator $\Delta+\kappa^{2}$
has a nontrivial kernel and to properly select solutions of 
\eqref{eq:helm0} additional conditions are needed.
It is natural to require asymptotic conditions at infinity,
and the standard one is the
\emph{Sommerfeld radiation condition}
\begin{equation}\label{eq:somm0}
  |x|^{\frac{n-1}{2}}\nabla(e^{-i \kappa|x|}v)\to0
  \quad\text{as}\quad 
  |x|\to \infty.
\end{equation}
Condition \eqref{eq:somm0} guarantees uniqueness for
\eqref{eq:helm0}, but it can be substantially relaxed as discussed
in the following.

The second part of the problem is the effective construction
of solutions;
this is usually done by taking $\kappa^{2}=\lambda+i \epsilon$
complex valued and letting $\epsilon\to0$. 
When the limit exists, one says that the
\emph{limiting absorption principle} holds.
Note that for $\kappa^{2} \not\in \mathbb{R}$
equation \eqref{eq:helm0} is the resolvent equation
$v=R(\kappa^{2})f$ for $R(z)=(z+\Delta)^{-1}$,
which is a bounded operator on $L^{2}$ if and only if 
$z\not\in \sigma(-\Delta)$. Thus the problem amounts to
estimate the resolvent operator $R(z)$
uniformly in $z\not\in \mathbb{R}$.
As a byproduct, one obtains that the resolvent operator
in the limits $\pm\Im z\to0$
extends to operators $R(\lambda\pm i0)$ which are
bounded between suitable weighted Sobolev spaces.

The Helmholtz equation with potential perturbations was
studied in 
\cite{agm},
\cite{agmhor},
where the correct functional setting for the problem was
established, and in
\cite{pertveg1} 
\cite{pertveg2},
where non decaying potentials were allowed.
More general Schr\"{o}dinger operators with electromagnetic
potentials were considered  in
\cite{barfan}
\cite{barruiveg},
\cite{barvegzub},
\cite{eidus},
\cite{fan},
\cite{ikesai},
\cite{zub1},
\cite{zub2}.
Uniform resolvent estimates in the case of variable coefficients 
were obtained in
\cite{MarzuolaMetcalfeTataru08-a},
\cite{rodtao},
\cite{Tataru08-a}
and the predecessor 
\cite{cacdanluc}
of this paper,
and estimates local in frequency for general elliptic
operators were proved in Chapter 30 of
\cite{Hormander94-b}.
We also mention the connection of resolvent estimates
with smoothing and Strichartz estimates for
the corresponding evolution equations
(exploited first in \cite{joursofsog}, \cite{Yajima90-a},
\cite{rodschlag}; see also 
\cite{DAnconaFanelli08-a}, \cite{CassanoDAncona15-a}
and the references in the papers mentioned above).

In recent years the problem of establishing sharp regularity
and decay conditions on the potentials has attracted some
attention, also in view of the applications to dispersive
equations. The critical threshold for electric potentials is
$\sim|x|^{-2}$ and for magnetic potentials $\sim|x|^{-1}$.
Uniform resolvent estimates for singular potentials of
critical decay were obtained in
\cite{BurqPlanchonStalker04-a},
\cite{FanelliVega09-a}
(see also \cite{DAnconaFanelliVega10-a}),
while the limiting absorption principle was studied in
\cite{zub1},
\cite{barvegzub}.

Our goal here is to study the interaction of singular
potentials with a nonflat metric which is a long range,
small perturbation of the euclidean metric. 
We consider the
following \emph{generalized Helmholtz equation} 
\begin{equation}\label{eq:helmh1}
  (L+\lambda+i \epsilon)v=f,
  \qquad
  \lambda,\epsilon\in \mathbb{R}
\end{equation}
where $L$ is an operator of the form
\begin{equation}\label{eq:opL}
  Lv=\nabla^b\cdot(a(x)\nabla^bv)+cv,
  \qquad
  \nabla^{b}=\nabla+ib,
\end{equation}
defined 
on the exterior $\Omega=\mathbb{R}^{n}\setminus \Sigma$
of a compact, possibly empty obstacle $\Sigma$ with $C^{1}$
boundary, in dimension $n\ge3$.
Here $a(x)=[a_{j k}(x)]_{j,k=1}^n$
is a real valued, positive definite 
symmetric matrix,
$b$ takes vaues in $\mathbb{R}^{n}$ and $c$ in $\mathbb{R}$.
We shall always assume that
\begin{equation}\label{eq:Lisslefadj}
  L\ \text{is selfadjoint with domain}\ 
  H^{2}(\Omega)\cap H^{1}_{0}(\Omega)
\end{equation}
i.e., we restrict to Dirichlet boundary conditions.
Note however that in the course of the paper we shall use the
same notation for the selfadjoint operator $L$ and the
differential operator \eqref{eq:opL} (which operates
also on functions outside $D(L)$, 
e.g. in weighted $L^{2}$ spaces).
We shall assume that the metric $a(x)$ is a small perturbation
of the flat metric, in an appropriate sense precised below,
so that in particular trapping is excluded.
Concerning the boundary $\partial \Omega$, we shall always assume
that it is \emph{starshaped with respect to the metric} $a(x)$:
this means
\begin{equation}\label{eq:astarsh}
  a(x)x\cdot\vec\nu(x)\leq 0
  \qquad
  \forall x\in \partial \Omega
\end{equation}
where $\vec\nu(x)$ is the exterior normal to $\Omega$
at $x\in\partial\Omega$. 

The assumptions on the magnetic potential 
$b(x)=(b_{1},\dots,b_{n})$ will be
expressed in terms of the corresponding field
\begin{equation*}
  db=[\partial_{j}b_{\ell}-\partial_{\ell}b_{j}]_{j,\ell=1}^{n}
\end{equation*}
as it is physically natural; actually it is sufficient
to impose bounds only on the \emph{tangential part} of $db$
for the metric $a(x)$, which is the vector
$\widehat{db}=(\widehat{db}_{1},\dots,\widehat{db}_{n})$
defined by
\begin{equation*}
  \widehat{db}(x)=db(x)a(x)\widehat{x}
  \quad\text{i.e.}\quad 
  \widehat{db}_{j}=
  (\partial_{j}b_{\ell}-\partial_{\ell}b_{j})a_{\ell m}
  \widehat{x}_{m},
  \qquad
  \textstyle
  \widehat{x}=\frac{x}{|x|}.
\end{equation*}
This fact was already noted in \cite{barvegzub} (see also \cite{cacvar}).
Here and in the following we use the
convention of implicit summation over repeated indices.
Note that for a vector $w\in \mathbb{C}^{n}$
we define its radial part $w_{R}$ and its tangential part
$w_{T}$ as
\begin{equation}\label{eq:tangder}
  w_{R}:=(\widehat{x}\cdot w)\widehat{x},
  \qquad
  w_{T}:=w-w_{R}
\end{equation}
respectively; we have of course $|w|^{2}=|w_{R}|^{2}+|w_{T}|^{2}$.

The relevant functional spaces for our problem are
the space $\dot Y$ with norm
\begin{equation*}
  \textstyle
  \|v\|_{\dot{Y}}^2:=\sup_{R>0}\frac1{R}
  \int_{\Omega\cap\{|x|\leq R\}}|v|^2dx
  \simeq
  \||x|^{-1/2}v\|_{\ell^{\infty}L^{2}}^{2}
\end{equation*}
and its (pre)dual space $\dot Y^{*}$ with norm
\begin{equation*}
  \|v\|_{\dot Y^{*}}\simeq
  \||x|^{1/2}v\|_{\ell^{1}L^{2}};
\end{equation*}
the notation $\ell^{p}L^{q}$ refers to the dyadic norms
\begin{equation}\label{eq:dyadicsp}
  \|v\|_{\ell^{p}L^{q}}
  :=
  \Bigl(
  \sum_{j\in \mathbb{Z}}
    \|v\|_{L^{q}(\Omega\cap\{2^{j}\le|x|<2^{j+1}\})}^{p}
  \Bigr)^{1/p},
\end{equation}
with obvious modification when $p=\infty$.
Note that $\dot Y^{*}$ is an homogeneous version
of the Agmon--H\"{o}rmander space $B$
(see \cite{agmhor}).
An important role will be played also by the space $\dot X$
with norm
\begin{equation*}
  \textstyle
  \|v\|_{\dot{X}}^2:=\sup_{R>0}\frac1{R^2}
  \int_{\Omega\cap\{|x|=R\}}|v|^2dS
\end{equation*}
where $dS$ is the surface measure on the sphere $|x|=R$.
Our main result is the following;
in the statement 
$|a(x)|$ denotes the operator norm of the matrix $a(x)$,
and we use the shorthand notation $|a'(x)|$ to denote
$\sum_{|\alpha|=1}|\partial^{\alpha}a(x)|$, and similarly for
$a'',a'''$, while 
$|b'(x)|=\sum_{|\alpha|=1}|\partial^{\alpha}b(x)|$.

\begin{theorem}[Limiting absorption principle]\label{the:existence}
  Let $n\ge3$, $\delta\in(0,1)$ and let
  $L$ and $\Omega$ be as in
  \eqref{eq:opL},\eqref{eq:Lisslefadj},\eqref{eq:astarsh}.
  There exist two constants $\overline{\kappa}>0$, 
  $\overline{\sigma}>0$ depending only on $n,\delta$
  such that the following holds.

  Assume that for some $\kappa\in[0,\overline{\kappa}]$ and $K\ge0$
  the coefficients of $L$ satisfy:
  \begin{enumerate}
  [noitemsep,topsep=0pt,parsep=0pt,partopsep=0pt,
  label=\textit{(\roman*)}]
    \item 
    $\|\bra{x}^{\delta}(|a-I|+|x||a'|)\|_{\ell^{1}L^{\infty}}
        <\infty$ and
    \begin{equation*}
      \||a-I|+|x||a'|\|_{\ell^{1}L^{\infty}}
      +|x|^{2}|a''|+|x|^{3}|a'''|\le \kappa.
    \end{equation*}

    \item $b',b^{2}\in L^{n,\infty}$ and $b=b_{S}+b_{L}$ with
    \begin{equation*}
      |x|^{2}|\widehat{db}_{S}|\le \kappa,
      \qquad
      \bra{x}^{\delta+1}|\widehat{db}_{L}|\le K.
    \end{equation*}
    When $n=3$ we assume the stronger condition
    $\||x|^{2}\widehat{db}_{S}\|_{\ell^{1}L^{\infty}}\le \kappa$.

    \item $c=c_{S}+c_{L}$ with
    $|x|^{2}c_{S}$, $|x|^{3}\nabla c_{S}\in L^{\infty}$ and
    \begin{equation*}
      \textstyle
      c_{S}\ge-\frac{\kappa}{|x|^{2}},
      \qquad
      -\partial_{r}(|x|c_{S})\ge-\frac{\kappa}{|x|^{2}},
      \qquad
      \bra{x}^{\delta}|c_{L}|\le K.
    \end{equation*}
  \end{enumerate}
  Then for $\lambda> \overline{\sigma}\cdot(K+K^{2})$
  and all $f$ with $\int|x|^{\delta}\bra{x}|f|^{2}<\infty$
  the equation
  \begin{equation}\label{eq:helmeq}
    (L+\lambda)v=f
  \end{equation}
  has a unique solution $v\in \dot Y\cap H^{2}_{loc}(\Omega)$
  satisfying $v\vert_{\partial\Omega}=0$ and
  the radiation condition
  \begin{equation}\label{eq:somm1}
    \liminf_{R\to+\infty}
    \int_{|x|=R}
    |\nabla ^{b}v-i \widehat{x}\lambda^{1/2}v|^{2}dS=0.
  \end{equation}
  In addition, the solution satisfies the smoothing estimate
  \begin{equation}\label{eq:newsmoo2}
    \textstyle
    \|v\|_{\dot X}
    +
    \lambda^{\frac12}\|v\|_{\dot Y}
    +
    \|\nabla^{b}v\|_{\dot Y}+
    \|(a \nabla^{b}v)_{T}\|_{L^{2}}+
    (n-3)
    \left\|\frac{v}{|x|^{3/2}}\right\|_{L^{2}}
    \le
    c(n)
    \|f\|_{\dot Y^{*}}
  \end{equation}
  and if $\epsilon_{k}\in \mathbb{R}\setminus \{ 0 \} $ is an
  arbitrary sequence with $\epsilon_{k}\to0$, then
  $v$ is the limit in $H^{1}_{loc}(\Omega)$ of the
  solutions $v_{k}\in H^{1}_{0}(\Omega)\cap H^{2}(\Omega)$ of
  \begin{equation*}
    (L+\lambda+i \epsilon_{k})v_{k}=f.
  \end{equation*}
\end{theorem}

When $K=0$, i.e., when the long range components 
$b_{L},c_{L}$ of the potentials are absent, the previous
result implies that the limiting absorption principle 
is valid for all values of $\lambda$
and for (short range) potentials with critical singularities,
provided suitable smallness conditions are assumed.
When $K\not=0$, i.e., if long range potentials are present,
we obtain a similar result but only for
large frequencies $\lambda$ depending on the size of the
potentials, which can be arbitrarily large.

The structure of the proof is the following:
\begin{itemize}
[noitemsep,topsep=0pt,parsep=0pt,partopsep=0pt]
  \item 
  The main tool used in the Theorem is a smoothing estimate for the
  resolvent $R(z)=(L+z)^{-1}$ outside the spectrum,
  proved in Section \ref{sec:smooest}
  (Theorem \ref{the:smoo}). The estimate improves on
  earlier results, notably on a similar estimate in the
  predecessor of this paper \cite{cacdanluc}.
  Indeed, we admit large potentials with critical singularities
  and the estimate is uniform for $\Re z\gg1$.
  In the short range case, 
  if $\widehat{db}_{S}$ and the negative part of $c_{S}$
  satisfy suitable
  smallness conditions, the estimate is uniform for all
  $z\in\mathbb{C}$. A few  applications
  include the non existence of embedded eigenvalues or
  resonances for $L$, and smoothing estimates
  for the Schr\"{o}dinger and wave flows associated to $L$.

  \item 
  The smoothing estimate alone is not sufficient to
  exclude functions in the
  kernel of $L+\lambda$. However, if the source term $f$
  has a slightly better decay, then the difference
  $\nabla^{b}v-i \widehat{x}\sqrt{\lambda}v$ satisfies a
  stronger estimate, and this is enough to deduce a
  weak Sommerfeld radiation condition and hence uniqueness
  of the solution. The radiation estimate is proved in
  Theorem \ref{the:sommlarge} in
  Section \ref{sec:radest}.

  \item 
  In the last Section \ref{sec:conclusionprf} we put
  together all the elements and prove the limiting
  absorption principle for $L$.
\end{itemize}

We conclude the Introduction by examining a few physically
interesting singular potentials 
to which the previous result can be applied.

\begin{remark}[Coulomb potential]\label{rem:examples}
  We can handle potentials of the form
  \begin{equation*}
    c(x)=\frac{C}{|x|^{a}},
    \qquad
    0<a\le2
  \end{equation*}
  including in particular the Coulomb potential $a=1$.
  In the critical case $a=2$, we must require in addition that
  $C\ge-\overline{\kappa}$ for a suitable $\overline{\kappa}\ge0$
  depending on $n$, however in this case the result is valid
  without restrictions on the frequency.
\end{remark}

\begin{remark}[Aharonov--Bohm]\label{rem:AB}
  Consider a magnetic potential $b(x)$ satisfying
  \begin{equation}\label{eq:homog}
    x \cdot b(x)=0
    \quad\text{and}\quad 
    b(tx)=t^{-1}b(x)
  \end{equation}
  for all $x\in \Omega$ and $t>0$ such that $tx\in \Omega$.
  The first condition is simply a choice of gauge, which is not
  restrictive, and the second one states
  that $b(x)$ is homogeneous of degree $-1$, 
  which is precisely the critical
  scaling for magnetic potentials.
  Then one checks easily (see \cite{barvegzub}) that
  \begin{equation*}
    db(x)\widehat{x}=0
    \quad\text{for all}\quad 
    x\in \Omega.
  \end{equation*}
  This implies
  \begin{equation*}
    \widehat{db}(x)=
    db(x)a(x)\widehat{x}=
    db(x)(a(x)-I)\widehat{x}
  \end{equation*}
  and as a consequence
  \begin{equation*}
    \||x|^{2}\widehat{db}(x)\|_{\ell^{1}L^{\infty}}\le
    \|a-I\|_{\ell^{1}L^{\infty}}\||x|^{2}b\|_{L^{\infty}}.
  \end{equation*}
  Since by homogeneity we have 
  also $\||x|^{2}b\|_{L^{\infty}}<\infty$, recalling that
  $\|a-I\|_{\ell^{1}L^{\infty}}$ is assumed to be 
  sufficiently small, we conclude that any 
  magnetic potential $b$
  satisfying \eqref{eq:homog} (or more generally, any potential
  $b=b_{S}+b_{L}$ with $b_{S}$ satisfying \eqref{eq:homog}
  and $b_{L}$ as in the Theorem)
  is covered by Theorem \ref{the:existence}.
  Interesting examples in $\mathbb{R}^{3}$ include
  the so called \emph{Aharonov--Bohm} potentials
  \begin{equation*}
    b(x)=C
    \left(
      \frac{-x_{2}}{x_{1}^{2}+x_{2}^{2}},
      \frac{x_{1}}{x_{1}^{2}+x_{2}^{2}},
      0
    \right)
  \end{equation*}
  and potentials of the form
  \begin{equation*}
    b(x)=C 
    \left(-\frac{x_{2}}{|x|^{2}},\frac{x_{1}}{|x|^{2}},0\right).
  \end{equation*}
  In both cases the result is valid for all frequencies,
  independently of the size or sign of $C$.
\end{remark}

\section{The smoothing estimate} \label{sec:smooest}

In this section we develop the key tool for
Theorem \ref{the:existence}: a smoothing estimate
for the resolvent of $L$ which is uniform on appropriate
regions of $\mathbb{C}$.
In order to get sharp results, we distinguish two situations:
\begin{enumerate}
  \item
  \emph{short range} perturbations of $\Delta$
  with critical singularities
  (Assumption (A$_{0}$)). In this case we can
  prove a uniform smoothing estimate for all 
  $z\in \mathbb{C}\setminus \mathbb{R}$;
  \item 
  \emph{long range} perturbations of $\Delta$, with large
  electromagnetic potentials of milder decay
  at infinity (Assumption (A)).
  In this case the estimate is uniform on
  a region $\Re z>C$, where $C$ is a suitable norm of the
  long range component of the potentials.
\end{enumerate}
Moreover, from our analysis one can read 
precisely the influence of
different components of the potentials $b$ and $c$ on the
estimate.

The assumptions in the short range case are the following:

\bigskip

\textsc{Assumption (A$_{0}$)}.
Let $n\ge3$ and let
$L$ and $\Omega$ be as in
\eqref{eq:opL},\eqref{eq:Lisslefadj},\eqref{eq:astarsh},
with $b',b^{2}\in L^{n,\infty}$.
We assume that, for some constant $\mu\ge0$
\begin{equation*}
  C_{a}(x):=
  |a-I|+|x||a'|+|x|^{2}|a''|+|x|^{3}|a'''|
  \le\mu,
  \qquad
  \||x|a'\|_{\ell^{1}L^{\infty}}\le\mu.
\end{equation*}
The magnetic field in dimension $n\ge4$ is of the form
$b=b_{I}+b_{II}$ and in dimension $n=3$ of the form $b=b_{I}$, with
\begin{equation*}
  \||x|^{2}\widehat{db}_{I}
    \|_{\ell^{1}L^{\infty}}+
  \||x|^{2}|a-I|\widehat{db}_{II}\|_{\ell^{1}L^{\infty}}+
  \||x|^{2}\widehat{db}_{II}\|_{L^{\infty}}
  \le\mu.
\end{equation*}
The electric field is of the form $c=c_{I}+c_{II}$ with
\begin{equation*}
 \||x|^{2}
 c_{II}\|_{\ell^{1} L^{\infty}}\le\mu,\quad c_{I,-}\in L^\infty
\end{equation*}
and in dimension $n\ge4$
\begin{equation*}
  |a-I|\cdot
  (|x|^{2} |c_{I}|+|x|^{3}|\nabla c_{I}|)+
  |x|^{2}\cdot(c_{I,-}+[\partial_{r}(rc_{I})]_{+}
  )
  \le\mu
\end{equation*}
while in dimension $n=3$
\begin{equation*}
  \|
  |a-I|\cdot
  (|x|^{2} |c_{I}|+|x|^{3}|\nabla c_{I}|)
  +
  |x|^{2}\cdot(c_{I,-}+[\partial_{r}(rc_{I})]_{+}
  )\|
    _{\ell^{1}L^{\infty}}
  \le\mu.
\end{equation*}

\bigskip

In the long range case the assumptions are the following.
Note that Assumption (A) reduces to (A$_{0}$)
when $Z=0$:

\bigskip

\textsc{Assumption (A)}.
We assume $b=b_{I}+b_{II}+b_{III}$
and $c=c_{I}+c_{II}+c_{III}+c_{IV}$
with $b_{I}$, $b_{II}$, $c_{I}$, $c_{II}$
as in (A$_{0}$) while
\begin{equation*}
  \||x|\widehat{db}_{III}\|_{\ell^{1}L^{\infty}}\le Z,
  \qquad
  \||x|\bra{x}^{-1}c_{IV}\|_{\ell^{1}L^{\infty}}\le Z,
\end{equation*}
\begin{equation*}
  \|
  |a-I|\cdot
  (|c_{III}|+|x||\nabla c_{III}|)
  +
  |x|^{2}\cdot(c_{III,-}+[\partial_{r}(rc_{III})]_{+}
    \|_{\ell^{1} L^{\infty}}
  \le Z.
\end{equation*}

\bigskip

Then we can prove:

\begin{theorem}[Smoothing estimate]\label{the:smoo}
  There exist two constants $\mu_{0}(n)$ and $c_{0}(n)$
  depending only on $n$ such that the following holds.

  Let $v\in H^{2}_{loc}(\Omega)$ with $v\vert_{\partial \Omega}=0$
  be such that
  \begin{equation}\label{eq:condinf}
    \liminf_{R\to \infty}
    \int_{|x|=R}(|\nabla^{b}v|^{2}+|v|^{2})dS=0
  \end{equation}
  and define for some $\lambda,\epsilon\in \mathbb{R}$
  \begin{equation*}
    f=(L+\lambda+i \epsilon)v.
  \end{equation*}
  If (A$_{0}$) holds with $\mu<\mu_{0}(n)$ then
  \begin{equation}\label{eq:newsmoo}
    \textstyle
    \|v\|_{\dot X}
    +
    (|\lambda|+|\epsilon|)^{1/2}\|v\|_{\dot Y}
    +
    \|\nabla^{b}v\|_{\dot Y}+
    \|(a \nabla^{b}v)_{T}\|_{L^{2}}+
    (n-3)
    \left\|\frac{v}{|x|^{3/2}}\right\|_{L^{2}}
    \le
    c(n)
    \|f\|_{\dot Y^{*}}.
  \end{equation}
  The same estimate is valid if (A) holds with
  $\mu<\mu_{0}(n)$ and $\lambda\ge c_{0}(n)(Z+Z^{2})$.
\end{theorem}

\begin{remark}[Uniform resolvent estimate]\label{rem:resest}
  Condition \eqref{eq:condinf} is satisfied
  if $v$ is in $H^{1}$.
  Thus the Theorem applies in particular to the solution $v$ of
  \begin{equation*}
    (L+\lambda+i \epsilon)v=f
  \end{equation*}
  for $\epsilon\neq0$ and $f\in L^{2}(\Omega)$,
  which exists and belongs to $H^{1}_{0}(\Omega)\cap H^{2}(\Omega)$
  by the assumptions on $L$. This gives the following 
  estimate for the resolvent operator $R(z)=(z+L)^{-1}$,
  uniform in $z\not\in \mathbb{R}$ or in
  $\Re z\ge c_{0}(n)(Z+Z^{2}),z\not\in \mathbb{R}$ respectively:
  \begin{equation*}
    \|\nabla^{b}R(z)f\|_{\dot Y}
    +
    |z|^{1/2}
    \|R(z)f\|_{\dot Y}
    +
    \|R(z)f\|_{\dot X}
    \lesssim\|f\|_{\dot Y^{*}}.
  \end{equation*}
\end{remark}

\begin{remark}[Absence of embedded eigenvalues or resonances]
  \label{rem:reson}
  Suppose $v$ is a solution of
  \begin{equation*}
    (L+\lambda)v=0,
    \qquad
    v\vert_{\partial \Omega}=0
  \end{equation*}
  for some $\lambda\ge c_{0}(n)(Z+Z^{2})$.
  From the smoothing estimate, we see that if
  $v$ satisfies condition \eqref{eq:condinf} then $v \equiv0$.

  Since any function $v\in H^{1}_{0}(\Omega)$ satisfies 
  condition \eqref{eq:condinf}, this implies that there is no
  eigenvalue $\lambda\ge c_{n}(Z+Z^{2})$.
  In particular in the case $Z=0$ (that is to say, 
  under condition (A$_{0}$)) we obtain there are no
  \emph{embedded eigenvalues} in the spectrum of  $L$.

  A similar argument gives a more general result about
  resonances. 
  Writing $\Omega_{\le R}=\Omega\cap\{|x|\le R\}$,
  we say that a function $v$ is a \emph{resonance at} 
  $z\in \mathbb{C}$ if
  \begin{equation*}
    \textstyle
    (L+z)v=0,
    \qquad
    v\vert_{\partial \Omega}=0,
    \qquad
    v\not \equiv0,
    \qquad
    \liminf_{R\to \infty}\frac1R\int_{\Omega_{\le R}}|v|^{2}=0.
  \end{equation*}
  Note that the last condition is weaker than the usual one:
  \begin{equation*}
    \textstyle
    \bra{x}^{-s}v\in L^{2}
    \quad\text{for some}\quad 
    s<\frac12
    \quad\implies\quad
    \lim_{R\to \infty}\frac1R\int_{\Omega_{\le R}}|v|^{2}=0.
  \end{equation*}
  Then we have
\end{remark}

\begin{corollary}[Absence of resonances]\label{cor:reson}
  Assume (A) holds with $\mu<\mu_{0}(n)$, and let 
  $\lambda\ge c_{0}(n)(Z+Z^{2})$. Then no resonance exists 
  at $\lambda$.
\end{corollary}

\begin{proof}
  We must only prove that $v$ satisfies condition \eqref{eq:condinf}.
  For $|v|^{2}$ this follows immediately from the assumption
  $\liminf\frac1R\int_{\Omega_{\le R}}|v|^{2}=0$. 
  For $|\nabla^{b}v|^{2}$,
  we apply Lemma \ref{lem:localbound} from
  Section \ref{sec:conclusionprf} which gives
  \begin{equation*}
    \textstyle
    \liminf\frac1R\int_{\Omega_{\le R}}|\nabla^{b} v|^{2}
    \lesssim
    \liminf\frac1R\int_{\Omega_{\le R}}|v|^{2}=0.
  \end{equation*}
\end{proof}

\begin{remark}[Smoothing estimates for dispersive flows]
  \label{rem:schrowave}
  A natural application of estimate \eqref{eq:newsmoo}
  to dispersive equations is given by
  Kato's theory of smoothing operators. We recall the procedure
  in the simplest case. 
  Assume (A$_{0}$) holds. Then, from
  \eqref{eq:newsmoo} we deduce the (Hilbert space) estimate
  \begin{equation*}
    \|\bra{x}^{-3/2-}v\|_{L^{2}}
    \lesssim
    \|\bra{x}^{1/2+}f\|_{L^{2}}
  \end{equation*}
  uniform in $\lambda+i \epsilon\not\in \mathbb{R}$,
  which can be written as the resolvent estimate
  \begin{equation*}
    \|\bra{x}^{-3/2-}R(z)f\|_{L^{2}}
    \lesssim
    \|\bra{x}^{1/2+}f\|_{L^{2}}
  \end{equation*}
  uniform in $z\not\in \mathbb{R}$. By duality and interpolation
  we get
  \begin{equation*}
    \|\bra{x}^{-1-}R(z)f\|_{L^{2}}
    \lesssim
    \|\bra{x}^{1+}f\|_{L^{2}}
    \quad\text{i.e.}\quad 
    \|A^{*}R(z)Af\|_{L^{2}}
    \lesssim
    \|f\|_{L^{2}}
  \end{equation*}
  where $A=\bra{x}^{-1-}$ is the multiplication operator.
  In terms of Kato's theory, this means that $A$ is 
  \emph{supersmoothing} for the operator $L$, and immediate
  consequences of the theory are the estimates for the
  Schr\"{o}dinger group $e^{itL}$
  \begin{equation*}
    \|\bra{x}^{-1-}e^{itL}f\|_{L^{2}_{t}L^{2}(\Omega)}
    \lesssim
    \|f\|_{L^{2}(\Omega)}
  \end{equation*}
  and the corresponding Duhamel term
  \begin{equation*}
    \textstyle
    \|\int_{0}^{t}\bra{x}^{-1-}e^{i(t-s)L}F(s)ds\|
        _{L^{2}_{t}L^{2}(\Omega)}
    \lesssim
    \|\bra{x}^{1+}F\|
        _{L^{2}_{t}L^{2}(\Omega)}.
  \end{equation*}
  Moreover, if $L$ is nonnegative, we also get the estimate
  for the wave flow $e^{it \sqrt{L}}$
  \begin{equation*}
    \|\bra{x}^{-1-}e^{it \sqrt{L}}f\|
        _{L^{2}_{t}L^{2}(\Omega)}
    \lesssim
    \|L^{1/4}f\|_{L^{2}(\Omega)}
  \end{equation*}
  and a similar one for the Duhamel term.
  With some more work, smoothing estimates
  with a weight $\bra{x}^{-1/2-}$ can be deduced for
  the flows $|D|^{1/2}e^{itL}$ and
  $|D|^{1/2}e^{it \sqrt{L}}$.
  For more details, and the extension of Kato's theory to
  the wave and Klein--Gordon groups, we refer to
  \cite{dan}.
\end{remark}

\subsection{Notations}\label{sub:notations}

With the convention of implicit summation over repeated indices,
we write
\begin{equation*}
  A^{b}v:=\nabla^b\cdot(a(x)\nabla^bv)=
  \partial_j^b(a_{jk}(x)\partial_k^bv),
  \qquad
  Av:=\nabla\cdot(a(x)\nabla v)=
  \partial_j(a_{jk}(x)\partial_kv).
\end{equation*}
We use the notations
\begin{equation*}
  \widehat{x}:=\frac{x}{|x|}=(\widehat{x}_{1},\dots,\widehat{x}_{n}),
  \qquad
  a(w,z):=a_{jk}(x)w_k\overline{z}_j,
  \qquad
  a_{jk;\ell}:=\partial_{\ell}a_{jk}
\end{equation*}
and
\begin{equation*}
  \widehat{a}(x):=a_{\ell m}(x)\widehat{x}_{\ell}\widehat{x}_{m},
  \qquad
  \overline{a}(x):=\mathop{tr}a(x)=a_{mm}(x),
  \qquad
  \widetilde{a}:=a_{\ell m;\ell}\widehat{x}_{m}.
\end{equation*}
If $a(x)$ is positive definite, we have
\begin{equation*}
  0\le \widehat{a}= a \widehat{x}\cdot \widehat{x}
  \le |a \widehat{x}|\le \overline{a}.
\end{equation*}
We shall use frequently the following identity,
valid for any radial function $\psi(x)=\psi(|x|)$:
\begin{equation}\label{eq:apsigen}
  A \psi(x)=\partial_{\ell}(a_{\ell m}\widehat{x}_{m}\psi')
  =
  \widehat{a}\psi''+
  \frac{\overline{a}-\widehat{a}}{|x|}
  \psi'
  +
  \widetilde{a}\psi'
\end{equation}
where $\psi'$ denotes the derivative of $\psi(r)$ with respect
to the radial variable. 

In order to refine the scale of dyadic spaces
$\ell^{p}L^{q}$, we introduce the
mixed radial-angular $L^{q}L^{r}$ norms on an
annulus $C=R_{1}\le|x|\le R_{2}$
\begin{eqnarray*}
  \textstyle
  \|v\|_{L^{q}L^{r}(C)}=
  \|v\|_{L^{q}_{|x|}L^{r}_{\theta}(C)}
  :&=&(\int_{R_{1}}^{R_{2}}
    (\int_{|x|=\rho}|v|^{r}dS)^{q/r}d\rho)^{1/q}\\
 &=&\left\|\|v\|_{L^{r}(|x|=\rho)}\right\|_{L^{q}(R_{1},R_{2};d\rho)},
\end{eqnarray*}
and on $\Omega\cap C$ we define
$\|v\|_{L^{q}L^{r}(\Omega\cap C)}=\|\one{\Omega}v\|_{L^{q}L^{r}}$.
When $q=r$ this definition reduces to the usual
$L^{q}(C)$ norm.
Then we define for all $p,q,r\in[1,\infty]$
\begin{equation}\label{eq:lpLp}
  \|v\|_{\ell^{p}L^{q}L^{r}}
  :=
  \|\{\|v\|_{L^{q}L^{r}(\Omega_{j})}\}
  _{j\in \mathbb{Z}}\|_{\ell^{p}},
  \qquad
  \Omega_{j}=\Omega\cap\{2^{j}\le|x|<2^{j+1}\}.
\end{equation}
In the case $q=r$ we reobtain the previous dyadic norms:
\begin{equation*}
  \|v\|_{\ell^{p}L^{q}}
  =
  \|v\|_{\ell^{p}L^{q}L^{q}}.
\end{equation*}
Both spaces $\dot X,\dot Y$ are included in this finer scale
\begin{equation}\label{eq:equivnorms}
  \|v\|_{\dot X}\simeq
  \||x|^{-1}v\|_{\ell^{\infty}L^{\infty}L^{2}},
  \qquad
  \|v\|_{\dot Y}\simeq
  \||x|^{-1/2}v\|_{\ell^{\infty}L^{2}}
\end{equation}
like the predual norm $\dot Y^{*}$, which is given by
\begin{equation*}
  \textstyle
  \|v\|_{\dot Y^{*}}\simeq
  \||x|^{1/2}v\|_{\ell^{1}L^{2}}
  \simeq
  \sum_{j\in \mathbb{Z}}
  2^{j/2}\|v\|_{L^{2}(C_{j}\cap \Omega)}.
\end{equation*}

\begin{remark}[Magnetic Hardy inequality]\label{rem:maghardy}
  We shall make frequent use of the \emph{magnetic Hardy inequality},
  valid for $s<n/2$ and $w\in H^{1}_{0}(\Omega)$:
  \begin{equation}\label{eq:maghardy}
    \textstyle
    \||x|^{-s}w\|_{L^{2}}
    \le
    \frac{2}{n-2s}
    \||x|^{1-s}\nabla^{b}w\|_{L^{2}}.
  \end{equation}
  This is proved as usual, starting from the identity
  \begin{equation*}
    \nabla \cdot
    \left\{
      \frac{\widehat{x}}{|x|^{2s-1}}|w|^{2}
    \right\}
    =
    2\Re \frac{w\;\widehat{x}\cdot \overline{\nabla w}}{|x|^{2s-1}}
    +
    \frac{n-2s}{|x|^{2s}}|w|^{2}
    =
    2\Re 
    \frac{w\;\widehat{x}\cdot \overline{\nabla^{b} w}}{|x|^{2s-1}}
    +
    \frac{n-2s}{|x|^{2s}}|w|^{2}
  \end{equation*}
  then integrating on $\Omega$, estimating with
  Cauchy--Schwartz
  \begin{equation*}
    \textstyle
    \int_{\Omega}
    \frac{n-2s}{|x|^{2s}}|w|^{2}dx
    \le
    \alpha \int_{\Omega}
    \frac{|w|^{2}}{|x|^{2s}}dx
    +
    \alpha^{-1}\int_{\Omega}
    \frac{|\nabla^{b}w|^{2}}{|x|^{2s-2}}
  \end{equation*}
  and finally optimizing the value of $\alpha$.
\end{remark}

\subsection{Basic identities and boundary terms}
\label{sub:virialid}

Using the two multipliers
\begin{equation*}
  [A^{b},\psi]\overline{v}=(A \psi)\overline{v}
  +2a(\nabla \psi,\nabla^{b}v)
  \quad\text{and}\quad 
  \phi \overline{v}
\end{equation*}
one obtains the following Morawetz type identities,
proved in \cite{cacdanluc}:

\begin{theorem}\label{the:morid}
  Let $v\in H^{2}_{loc}(\Omega)$
  on an open set $\Omega \subseteq \mathbb{R}^{n}$,
  $\lambda,\epsilon\in \mathbb{R}$,
  $a(x):\Omega\to \mathbb{R}^{n \times n}$ symmetric,
  $b(x):\Omega\to \mathbb{R}^{n}$ and
  $c,\phi,\psi:\Omega\to \mathbb{R}$ sufficiently smooth,
  and let
  \begin{equation}\label{eq:deff}
    f:=A^{b}v-c(x)v+(\lambda+i \epsilon)v.
  \end{equation}
  Then the following identity holds:
  \begin{equation}\label{eq:morid}
    I_{\nabla v}+I_{v}+I_{\epsilon}+I_{b}+I_{f}=
    \Re \partial_{j}\{Q_{j}+P_{j}\}
  \end{equation}
  where
  \begin{equation}\label{eq:If}
    I_{f}=
    \Re [(A\psi+\phi)\overline{v}f+2a(\nabla\psi,\nabla^bv)f],
  \end{equation}
  \begin{equation}\label{eq:Inablav}
    I_{\nabla v}=
    \alpha_{\ell m} 
    \Re(\partial^{b}_{m}v\ \overline{\partial^{b}_{\ell}v})
        +a(\nabla^{b}v,\nabla^{b}v)\phi,
    \quad
    \alpha_{\ell m}:=
      2a_{jm}\partial_{j}(a_{\ell k}
      \partial_{k}\psi)
      -a_{jk}\partial_{k}\psi \partial_{j}a_{\ell m}
  \end{equation}
  \begin{equation}\label{eq:Iv}
    \textstyle
    I_{v}=
    -\frac12 A(A \psi+\phi)|v|^{2}
    -[a(\nabla \psi,\nabla c)-c \phi+\lambda \phi]|v|^{2}
  \end{equation}
  \begin{equation}\label{eq:Iepsilon}
    I_{\epsilon}=2 \epsilon\Im[a(\nabla \psi,\nabla^{b}v)v]
  \end{equation}
  \begin{equation}\label{eq:Ib}
    I_{b}=2\Im[a_{jk}\partial^{b}_{k}v
       (\partial_{j}b_{\ell}-\partial_{\ell}b_{j})
       a_{\ell m}\partial_{m}\psi\ \overline{v}]
    =
    2\Im [(a\nabla^{b}v)\cdot (db \ a\nabla \psi)\overline{v}]
  \end{equation}
  and
  \begin{equation*}
    Q_{j}=
    a_{jk}\partial^{b}_{k}v \cdot 
      [A^{b},\psi]\overline{v}
      -\textstyle\frac12 a_{jk}(\partial_{k}A \psi)|v|^{2}
      -a_{jk}\partial_{k}\psi 
      \left[(c-\lambda)|v|^{2}+a(\nabla^{b}v,\nabla^{b}v)\right]
  \end{equation*}
  \begin{equation*}
    P_{j}=
    a_{jk}\partial^{b}_{k}v\phi \overline{v}
       -\textstyle \frac12 a_{jk}\partial_{k}\phi|v|^{2}.
  \end{equation*}
  Moreover we have the identity
  \begin{equation}\label{eq:morid2}
    \partial_{j}P_{j}=
    a(\nabla^{b}v,\nabla^{b}v)\phi
    +(c-\lambda-i \epsilon)|v|^{2}\phi
    +f \overline{v}\phi
    -\textstyle \frac12 A \phi|v|^{2}
    +i\Im a(\nabla^{b} v,v \nabla \phi).
  \end{equation}
\end{theorem}

\begin{remark}[Boundary terms]\label{rem:bdry}
  In the next computations we shall integrate identities
  \eqref{eq:morid} and \eqref{eq:morid2} on $\Omega$, with
  various choices of real valued 
  weights $\phi$ and $\psi$, with $\psi$ radial, for a function
  $v\in H^{2}_{loc}(\Omega)$ vanishing at $\partial \Omega$
  and satisfying the asymptotic condition \eqref{eq:condinf}.
  The weights will always be piecewise smooth functions,
  with possible singularities only at 0 or along spheres $|x|=R$;
  the worst singularity at 0 appearing in all computations is
  dominated by $|x|^{-3}$ in dimension $n\ge4$ and by $|x|^{-2}$
  in dimension $n=3$;
  concerning the singularity appearing along the sphere,
  in the worst case it will be a surface measure $\delta_{|x|=R}$ with
  a definite sign. Moreover, in our choice of
  $\psi$ we have $\psi'\in L^{\infty}$ and $\psi'\ge0$
  (see \eqref{eq:ourpsi} below).

  In order to handle the boundary terms, 
  some smoothness of the coefficients is necessary.
  We note that from our assumptions it follows that
  $a,a',a'',a''',c$ are bounded for large $x$ and
  \begin{equation}\label{eq:basicreg}
    a,|x|a',|x|^{2}a'',|x|^{3}a'''\in L^{\infty},
    \qquad
    b\in L^{n/2,\infty},
    \qquad
    b',c\in L^{n,\infty}.
  \end{equation}
  Then one checks easily that for 
  $v\in H^{2}_{loc}(\Omega)$
  the quantities $Q_{j}$ and $P_{j}$ are in $L^{1}_{loc}$,
  using the Sobolev-Lorentz embedding 
  $H^{1}\hookrightarrow L^{2}\cap L^{\frac{2n}{n-2},2}$
  which implies $|v|^{2}\in L^{1}\cap L^{\frac{n}{n-2},1}$,
  and the H\"{o}lder-Lorentz inequality.

  We integrate the identities \eqref{eq:morid} and \eqref{eq:morid2}
  on a set $\Omega\cap\{|x|\le M\}$ and let $M\to \infty$.
  At the boundary $\Omega\cap \{|x|=M\}$ we get the quantities
  \begin{equation*}
    \textstyle
    \int_{\Omega_{=M}}
    \nu_{j}Q_{j}dS,
    \qquad
    \int_{\Omega_{=M}}
    \nu_{j}P_{j}dS,
  \end{equation*}
  where $\vec{\nu}=(\nu_{1},\dots,\nu_{n})$ 
  is the exterior normal
  and $dS$ is the surface measure on the sphere $\{|x|= M\}$.
  Letting $M\to \infty$ along a suitable subsequence,
  by condition \eqref{eq:condinf} we see that both integrals
  tend to 0.
  Moreover, at the boundary $\partial \Omega$ one has directly
  $P_{j}\vert_{\partial \Omega}=0$ since $v\vert_{\partial \Omega}=0$.
  Concerning $Q_{j}$, after canceling the terms containing a factor $v$
  and noticing that $\nabla^{b}v=\nabla v+ibv=\nabla v$ on 
  $\partial \Omega$, we are left with
  \begin{equation}
    \textstyle
    \int_{\Omega}
    \partial_{j}Q_{j}
    =
    \int_{\partial \Omega}
    \left[
      2a(\nabla v,\vec{\nu})\cdot
      a(\widehat{x},\nabla v)
      -
      a(\nabla v,\nabla v)\cdot
      a(\widehat{x},\vec{\nu})
    \right]\psi'
    dS
  \end{equation}
  where $\vec{\nu}$ is the exterior unit normal to 
  $\partial \Omega$.
  Dirichled boundary conditions imply that
  $\nabla v$ is normal to $\partial \Omega$ so that
  $\nabla v=(\vec{\nu}\cdot \nabla v)\vec{\nu}$
  and hence
  \begin{equation*}
    a(\nabla v,\vec{\nu})=(\vec{\nu}\cdot \nabla v)
    a(\vec{\nu},\vec{\nu}),
    \quad
    a(\widehat{x},\nabla v)=
    (\vec{\nu}\cdot \nabla \overline{v})a(\widehat{x},\vec{\nu}),
    \end{equation*}
    \begin{equation*}
    a(\nabla v,\nabla v)=
    |\vec{\nu}\cdot \nabla \overline{v}|^{2}a(\vec{\nu},\vec{\nu})
  \end{equation*}
  and
  \begin{equation*}%
    \textstyle
    \int_{\Omega}\partial_{j}\Re Q_{j}
    =
    \int_{\partial \Omega}
    |\vec{\nu}\cdot \nabla v|^{2}
    a(\vec{\nu},\vec{\nu})
    a(\widehat{x},\vec{\nu})\psi'
    dS.
  \end{equation*}
  Now using the condition that $\partial \Omega$
  is $a(x)$-starshaped and recalling that
  $\psi'\ge0$ we conclude
  \begin{equation}\label{eq:bdryQ}
    \textstyle
    \int_{\Omega}\partial_{j}\Re Q_{j}\le0.
  \end{equation}
\end{remark}

\subsection{Preliminary estimates}\label{sub:prelim_est}

The first group of estimates is based on 
the identity \eqref{eq:morid2}.

\begin{lemma}[$I_{\epsilon}$]\label{lem:1}
  We have the identities
  \begin{equation}\label{eq:new1}
    \textstyle
    \epsilon\int_{\Omega}|v|^{2}=
    \Im \int_{\Omega} f \overline{v},
    \qquad
    \int_{\Omega}a(\nabla^{b}v,\nabla^{b}v)=
    \lambda\int_{\Omega}|v|^{2}-\Re\int_{\Omega}f \overline{v}-
    \int_{\Omega}c|v|^{2}.
  \end{equation}
  Moreover if we assume $\|a-I\|_{L^{\infty}}\le1/2$
  and $c=c_{I}+c_{II}$ with $c_{I,-}\in L^\infty$ and
  $\||x|^{2}c_{II,-}\|_{L^{\infty}}\le\frac{n-2}8$,
  we have the following estimate of the quantity
  $I_{\epsilon}:=2 \epsilon\Im[a(v\nabla \psi,\nabla^{b}v)]$
  \begin{equation}\label{eq:new2}
    \textstyle
    \int_{\Omega}
    |I_{\epsilon}|\le
    \sigma
    (|\lambda|+|\epsilon|+\|c_{I,-}\|_{L^{\infty}})
    \|v\|_{\dot Y}^{2}
    +
    C \sigma^{-1} \|f\|_{\dot Y^{*}}
  \end{equation}
  where $C=C(n,\|\nabla \psi\|_{L^{\infty}})$
  and $\sigma\in(0,1)$ is arbitrary.
  \end{lemma}

  \begin{proof}
  Consider identity \eqref{eq:morid2} with $\phi=1$
  and $c=0$ and let $g=A^{b}v+(\lambda+i \epsilon)v$,
  so that $g=f+c(x)v$.
  Taking the imaginary part 
  we get
  \begin{equation}\label{eq:Imphi}
    \epsilon|v|^{2}=
    \Im(g \overline{v})
    -\Im \partial_{j}\{\overline{v}a_{jk} 
    \partial_{k}^{b}v\}
  \end{equation}
  and integrating on $\Omega$ we obtain the first identity
  in \eqref{eq:new1}, since
  $\Im(f \overline{v})=\Im(g \overline{v})$.
  Note that the identity implies
  \begin{equation}\label{eq:new0}
    |\epsilon|\|v\|_{L^{2}}^{2}\le\|f \overline{v}\|_{L^{1}}.
  \end{equation}
  If instead we take the real part of \eqref{eq:morid2} with $\phi=1$
  and $c=0$ we get
  \begin{equation*}
    a(\nabla^{b}v,\nabla^{b}v)=
    \lambda|v|^{2}
    -\Re(g \overline{v})
    +\Re \partial_{j}\{\overline{v}a_{jk} \partial^{b}_{k}v\}.
  \end{equation*}
  Integrating on $\Omega$, the boundary term vanishes
  (see Remark \ref{rem:bdry}),
  and we get the second identity \eqref{eq:new1},
  after replacing $g=f+c(x)v$.

  We can now write
  \begin{equation*}
    \textstyle
    -\int_{\Omega}c|v|^{2}\le
    \int_{\Omega}c_{I,-}|v|^{2}
    +
    \int_{\Omega}c_{II,-}|v|^{2}
    \le
    \int_{\Omega}c_{I,-}|v|^{2}
    +\||x|^{2}c_{II,-}\|_{L^{\infty}}
    \int_{\Omega}\frac{|v|^{2}}{|x|^{2}}
  \end{equation*}
  and by the magnetic Hardy inequality \eqref{eq:maghardy}
  \begin{equation*}
    \textstyle
    \||x|^{2}c_{II,-}\|_{L^{\infty}}
    \int_{\Omega}\frac{|v|^{2}}{|x|^{2}}
    \le
    \frac{2\||x|^{2}c_{II,-}\|_{L^{\infty}}}{(n-2)}
    \int_{\Omega}|\nabla^{b}v|^{2}
    \le
    \frac12
    \int_{\Omega}
    a(\nabla^{b}v,\nabla^{b}v)
  \end{equation*}
  provided $\|a-I\|_{L^{\infty}}\le1/2$ and 
  $\||x|^{2}c_{II,-}\|_{L^{\infty}}\le \frac{n-2}{8}$.
  Absorbing the last term at the left hand side
  of \eqref{eq:new1} we have proved
  \begin{equation}\label{eq:new2b}
    \textstyle
    \int_{\Omega}a(\nabla^{b}v,\nabla^{b}v)\le
    2\lambda\int_{\Omega}|v|^{2}
    -
    2\Re\int_{\Omega}f \overline{v}
    +
    2\int_{\Omega}c_{I,-}|v|^{2}.
  \end{equation}
  Next, by Cauchy--Schwartz and $a\le NI$ we have
  \begin{equation*}
    |I_{\epsilon}|\le
    |v|
    |\epsilon|a(\nabla\psi,\nabla\psi)^{1/2}a(\nabla^{b}v,\nabla^{b}v)^{1/2}\le
    N^{1/2}  \|\nabla \psi\|_{L^{\infty}}
    |\epsilon||v|a(\nabla^{b}v,\nabla^{b}v)^{1/2}
  \end{equation*}
  and using \eqref{eq:new1}, \eqref{eq:new2b},
  with $C=2N^{1/2}  \|\nabla \psi\|_{L^{\infty}}$,
  \begin{equation*}
    \textstyle
    \int_{\Omega}|I_{\epsilon}|\le
    C
    \left[
    (\sgn \epsilon)
    \Im \int_{\Omega} f\overline{v}
    \right]^{1/2}
    \left[
    |\epsilon|
    \lambda\int_{\Omega}|v|^{2}
    -
    |\epsilon|\Re\int_{\Omega}f \overline{v}
    +
       |\epsilon|\int_{\Omega}c_{I,-}|v|^{2}
    \right]^{1/2}
  \end{equation*}
  (note that both quantities inside brackets are positive).
  Using again \eqref{eq:new0} we get
  \begin{equation*}
    \textstyle
    \int_{\Omega}
    |I_{\epsilon}|\le
    C
    \textstyle
    \left|
    \int_{\Omega} f \overline{v}
    \right|^{1/2}
    \left[
    (|\lambda|+|\epsilon|)|\int_{\Omega} f \overline{v}|
    +
    |\epsilon|\int_{\Omega}c_{I,-}|v|^{2}
    \right]^{1/2}
  \end{equation*}
  which implies
  \begin{equation*}
    \textstyle
    \int_{\Omega}
    |I_{\epsilon}|\le
    C
    (|\lambda|+|\epsilon|)^{1/2}
    \|f \overline{v}\|_{L^{1}}
    +
    C
    |\epsilon|^{1/2}
    \|f \overline{v}\|_{L^{1}}^{1/2}
    \|c_{I,-}^{1/2}v\|_{L^{2}}.
  \end{equation*}
  Using \eqref{eq:new0} we have
  \begin{equation*}
    |\epsilon|^{1/2}
    \|c_{I,-}^{1/2}v\|_{L^{2}}\le
    |\epsilon|^{1/2}
    \|c_{I,-}\|_{L^{\infty}}^{1/2}\|v\|_{L^{2}}\le
    \|c_{I,-}\|_{L^{\infty}}^{1/2}|
    \|f \overline{v}\|_{L^{1}}^{1/2};
  \end{equation*}
  plugging it into the previous inequality we get
  \begin{equation*}
    \textstyle
    \int_{\Omega}|I_{\epsilon}|\le
    C(|\lambda|+|\epsilon|+\|c_{I,-}\|_{L^{\infty}})^{1/2}
    \|f \overline{v}\|_{L^{1}}
  \end{equation*}
  and using Cauchy--Schwartz we obtain \eqref{eq:new2}.
\end{proof}

\begin{lemma}[Auxiliary estimate I]\label{lem:epsilon}
  We have
  \begin{equation}\label{eq:new3}
    |\epsilon|^{1/2}\|v\|_{\dot Y}
    \le
    C\|a\|_{L^{\infty}}
    \left(\|\nabla^{b} v\|_{\dot Y}+\|v\|_{\dot X}
    +\|f\|_{\dot Y^{*}}\right)
  \end{equation}
  for some universal constant $C$.
\end{lemma}

\begin{proof}
  Take the imaginary part of \eqref{eq:morid2}
  and choose $\phi$ as follows:
  \begin{equation}\label{eq:choicephi}
    \textstyle
    \phi(x)=1 \ \text{if}\ |x|\le R,\quad
    \phi(x)=2-\frac{|x|}{R} \ \text{if}\ R\le|x|\le2R,\quad
    \phi(x)=0 \ \text{if}\ |x|\ge 2R.
  \end{equation}
  Integrating on $\Omega$ the boundary term vanishes and we get
  \begin{equation}\label{eq:interm}
  \begin{split}
    \textstyle
    |\epsilon|
    \int_{\Omega_{\le R}}
    |v|^{2}
    \le &
    \textstyle
    \int_{\Omega_{\le2R}}
    |f \overline{v}|
    +\frac{N}{R}
    \int_{\Omega_{R\le|x|\le 2R}}
    |v||\nabla^{b} v|
      \\
    \le &
    2R\|f\|_{\dot Y^{*}}\|v\|_{\dot X}+
    3NR\|v\|_{\dot X}\|\nabla^{b} v\|_{\dot Y}.
  \end{split}
  \end{equation}
  Dividing by $R$ and taking the sup for $R>0$ we obtain \eqref{eq:new3}.
\end{proof}


\begin{lemma}[Auxiliary estimate II]\label{lem:lambdameno}
  Assume $\lambda=-\lambda_{-}\le0$ and
  $\|a-I\|_{L^{\infty}}+\||x|a'\|_{L^{\infty}}\le 1/8$.
  Then in dimension $n\ge4$ we have
  \begin{equation}\label{eq:new4b}
    \lambda_{-}\|v\|_{\dot Y}^{2}\le
    C\|c_{-}|x|^{2}\|_{L^{\infty}}
    \||x|^{-3/2}v\|_{L^{2}}^{2}+\delta\|v\|_{\dot X}^{2}
    + C\delta ^{-1}\|f\|_{\dot Y^{*}}
  \end{equation}
  and in all dimensions $n\ge3$ we have
  \begin{equation}\label{eq:new4c}
    \lambda_{-}\||x|^{-1/2}v\|_{L^{2}}^{2}
    +
    \||x|^{-1/2}\nabla^{b}v\|_{L^{2}}^{2}
    \le
    (C\||x|^{2}c_{-}\|_{\ell^{1}L^{\infty}}
    + \delta)
    \|v\|_{\dot X}^{2}
    + C\delta ^{-1}\|f\|_{\dot Y^{*}}^{2}
  \end{equation}
  for some universal constant $C$ and all $\delta\in(0,1)$.
  Note also that
  \begin{equation*}
    \|v\|_{\dot Y}\le\||x|^{-1/2}v\|_{L^{2}}.
  \end{equation*}
\end{lemma}

\begin{proof}
  Since $\lambda=-\lambda_{-}\le0$,
  we can rewrite (the real part of) \eqref{eq:morid2} in the form
  \begin{equation}\label{eq:tempident}
    (c_{+}+\lambda_{-})|v|^{2}\phi+
    a(\nabla^{b}v,\nabla^{b}v)\phi=
    \partial_{j}\Re P_{j}
    +c_{-}|v|^{2}\phi
    -\Re(f \overline{v})\phi
    +\textstyle \frac12 A \phi|v|^{2}.
  \end{equation}
  We choose the radial weight
  \begin{equation*}
    \textstyle
    \phi=\frac{1}{|x|\vee R}
    \quad\implies\quad
    \phi'=-\frac{1}{|x|^{2}}\one{|x|>R},
    \quad
    \phi''=-\frac{1}{R^{2}}\delta_{|x|=R}+
      \frac{2}{|x|^{3}}\one{|x|>R}.
  \end{equation*}
  By the formula 
  $A \phi=\widehat{a}\phi''+\frac{\bar{a}-\widehat{a}}{|x|}\phi'
    +\widetilde{a}\phi'$,
  writing $\widehat{a}=1+(a-I)\widehat{x}\cdot \widehat{x}$ and
  $\bar{a}=n+tr(a-I)$ and dropping a negative term, we get
  \begin{equation*}
    \textstyle
    A \phi=
      -\frac{\widehat{a}}{R^{2}}\delta_{|x|=R}+
      \frac{3 \widehat{a}-\bar{a}+|x|\widetilde{a}}{|x|^{3}}\one{|x|>R}
    \le
    \frac{3-n+(n+3)(|a-I|+|x||a'|)}{|x|^{3}}\one{|x|>R}.
  \end{equation*}
  In dimension $n\ge4$, if 
  $\|a-I\|_{L^{\infty}}+\||x|a'\|_{L^{\infty}}\le 1/6$, we get
  $A\phi\le0$; hence integrating \eqref{eq:tempident} on $\Omega$ and estimating $a(\nabla^{b}v,\nabla^{b}v)\ge \nu|\nabla^{b}v|^{2}$, we get
  \begin{equation*}
    \textstyle
    \int_{\Omega}
    \frac{(c_{+}+\lambda_{-})|v|^{2}+\nu|\nabla^{b}v|^{2}}{|x|\vee R}
    \le
    \int_{\Omega}
    \frac{c_{-}|v|^{2}+|f \bar{v}|}{|x|\vee R}.
  \end{equation*}
  Taking the sup over $R>0$ we conclude
  \begin{equation*}
    \textstyle
    \|c^{1/2}_{+}|x|^{-1/2}v\|_{L^{2}}^{2}
    +
    \lambda_{-}
    \||x|^{-1/2}v\|_{L^{2}}^{2}
    +
    \nu
    \||x|^{-1/2}\nabla^{b}v\|_{L^{2}}^{2}
    \le
    \|c^{1/2}_{-}|x|^{-1/2}v\|_{L^{2}}^{2}
    +
    \int_{\Omega}\frac{|f \overline{v}|}{|x|}.
  \end{equation*}
  Since $\|v\|_{\dot Y}\lesssim\||x|^{-1/2}v\|_{L^{2}}$, we have in
  particular
  \begin{equation*}
    \lambda_{-}\|v\|_{\dot Y}^{2}\le
    C
    \|c_{-}|x|^{2}\|_{L^{\infty}}
    \||x|^{-3/2}v\|_{L^{2}}^{2}+
    \||x|^{-1}v\|_{\dot Y}\|f\|_{\dot Y^{*}}
  \end{equation*}
  and using the inequality $\||x|^{-1}v\|_{\dot Y}\lesssim\|v\|_{\dot X}$
  we obtain \eqref{eq:new4b}.

  If the dimension is $n\ge3$ we choose a different weight,
  for $\sigma>0$ arbitrary:
  \begin{equation*}
    \textstyle
    \phi=\frac{1}{\sigma+|x|}
    \quad\implies\quad
    \frac12 A \phi=\frac{\widehat{a}}{(\sigma+|x|)^{3}}
      -\frac{\bar{a}-\widehat{a}+|x|\widetilde{a}}{(\sigma+|x|)^{2}|c|}.
  \end{equation*}
  By the estimates $\widehat{a}\le 1+C'_{a}$,
  $|x||\widetilde{a}|\le C'_{a}$, $\bar{a}\ge n(1-C'_{a})$
  with $C'_{a}=|a-I|+|x||a'|$, we have
  \begin{equation*}
    \textstyle
    \frac12 A \phi
    \le
    -
    \frac{|x|(n-2-(n+1)C'_{a})-\epsilon(n-1-nC'_{a})}{|x|(\sigma+|x|)^{3}}
    \le
    -\frac{1}{2|x|(\sigma+|x|)^{2}}
  \end{equation*}
  provided we choose e.g.~$C'_{a}\le1/8$. Hence 
  integrating \eqref{eq:tempident}
  on $\Omega$ and using again that 
  $a(\nabla^{b}v,\nabla^{b}v)\ge \nu|\nabla^{b}v|^{2}$,
  $\nu\ge \frac12$, we get for some universal constant $C$
  \begin{equation*}
    \textstyle
    \int_{\Omega}
    \frac{\lambda_{-}|v|^{2}+\nu|\nabla^{b}v|^{2}}{\sigma+|x|}
    +
    \int_{\Omega}
    \frac{|v|^{2}}{|x|(\sigma+|x|)^{2}}
    \le
    C
    \||x|^{-1/2}c_{-}^{1/2}v\|_{L^{2}}^{2}
    +
    C\||x|^{-1}f \overline{v}\|_{L^{1}}
  \end{equation*}
  \begin{equation*}
    \textstyle
    \le
    C\||x|^{2}c_{-}\|_{\ell^{1}L^{\infty}}\|v\|_{\dot X}^{2}
    +
    C\|f\|_{\dot Y^{*}}\|v\|_{\dot X}.
  \end{equation*}
  Letting $\sigma\to0$ we obtain \eqref{eq:new4c}.
\end{proof}

\begin{lemma}[Auxiliary estimate III]\label{lem:invsq}
  Let $n\ge4$.
  Assume $\|C_{a}\|_{L^{\infty}}+\||x|^{2}c_{-}\|_{L^{\infty}}\le1/16$.
  Then
  \begin{equation}\label{eq:new01}
    \||x|^{-1/2}\nabla^{b}v\|_{L^{2}(|x|\le1)}^{2}
    \le
    \lambda_{+}\||x|^{-1/2}v\|_{L^{2}(|x|\le2)}^{2}
    +
    c(n)\|v\|_{\dot X}^{2}+c(n)\|f\|_{\dot Y^{*}}^{2}.
  \end{equation}
  Note also that 
  $\lambda_{+}\||x|^{-1/2}v\|_{L^{2}(|x|\le2)}\le
   2\lambda_{+}\||x|^{-3/2}v\|_{L^{2}(|x|\le2)}$.
\end{lemma}

\begin{proof}
  Choose a smooth nonnegative weight of the form
  \begin{equation*}
    \phi=|x|^{-1}
    \ \text{for}\ |x|\le1,
    \qquad
    0\le\phi\le|x|^{-1}
    \ \text{for}\ 1\le|x|\le2,    
    \qquad
    \phi=0
    \ \text{for}\ |x|\ge2    
  \end{equation*} 
  in \eqref{eq:morid2},
  take the real part and integrate on $\Omega$.
  We get
  \begin{eqnarray*}
    \textstyle
    \int_{\Omega_{|x|\le1}}
    \frac{a(\nabla^{b}v,\nabla^{b}v)}{|x|}
   & \le&
    \int_{\Omega_{|x|\le2}}
    \left(
    (\lambda_{+}+c_{-})\frac{|v|^{2}}{|x|}
    -\frac{\bar{a}-3 \widehat{a}+|x| \widetilde{a}}{|x|^{3}}|v|^{2}
    +
    \frac{1}{|x|}
    |f \overline{v}|
    \right)
    \\
    &+&C\|v\|_{L^{2}(1\le|x|\le2)}^{2}
  \end{eqnarray*}
  for some $C=C(n,\|C_{a}\|_{L^{\infty}})$.
  Since
  \begin{equation}\label{AlsoRec}
    \textstyle
    \bar{a}-3 \widehat{a}+|x| \widetilde{a}\ge
    n-3-(n+4)\|C_{a}\|_{L^{\infty}}\ge \frac12
  \end{equation}
  if e.g.~$\|C_{a}\|_{L^{\infty}}\le1/16$, 
  and moreover
  \begin{equation*}
    \textstyle
    \int_{\Omega_{|x|\le2}}\frac{|f||v|}{|x|}
    \le
    \delta\||x|^{-3/2}v\|_{L^{2}(|x|\le2)}^{2}
    +
    \delta^{-1}\|f\|_{\dot Y^{*}}^{2},
    \qquad
    \|v\|_{L^{2}(1\le|x|\le2)}
    \le2\|v\|_{\dot X},
  \end{equation*}

  we have
  \begin{equation*}
  \begin{split}
    \textstyle
    \int_{\Omega_{|x|\le1}}
    \frac{a(\nabla^{b}v,\nabla^{b}v)}{|x|}
    \le
    \lambda_{+}
    \|\frac{v}{|x|^{1/2}}\|_{L^{2}(|x|\le2)}^{2}
    +
    &
    \textstyle
    (\delta+\||x|^{2}c_{-}\|_{L^{\infty}}-\frac14)
    \|\frac{v}{|x|^{3/2}}\|_{L^{2}(|x|\le2)}^{2}
    \\
    +&
    \textstyle
    C(n,\|C_{a}\|_{L^{\infty}})
    (\|v\|_{\dot X}^{2}+\delta^{-1}\|f\|_{\dot Y^{*}}^{2})
  \end{split}
  \end{equation*}
  by taking $\delta$ sufficiently small we get the claim.
\end{proof}

We recall the notations
\begin{equation*}
  (a \nabla^{b}v)_{R}=(\widehat{x}\cdot a \nabla^{b}v)\widehat{x},
  \qquad
  (a \nabla^{b}v)_{T}=a \nabla^{b}v-(a \nabla^{b}v)_{R}
\end{equation*}
for the radial and the tangential part of $a \nabla^{b}v$.
Note that in case the weight $\psi=\psi(|x|)$ is a radial
function, the term $I_{b}$ takes the form
\begin{equation*}
  I_{b}
  =
  2\Im [(a\nabla^{b}v)\cdot (db \ a \widehat{x})\ \overline{v}]\psi'
  =
  2\Im [(a\nabla^{b}v)\cdot \widehat{db}\ \overline{v}]\psi'
\end{equation*}
where $\widehat{db}:=db \ a \widehat{x}$ is the 
tangential part of the magnetic field.

\begin{lemma}[$I_{b}$]\label{lem:magnetic}
  Assume $\psi$ is a radial function, $b=b_{I}+b_{II}+b_{III}$.
  Then we have
  \begin{equation*}
    \textstyle
    \int_{\Omega}|I_{b}|
    \le
    C
    \beta_{1}
    \|\nabla^{b}v\|_{\dot Y}
    \|v\|_{\dot X}
    +
    C
    \beta_{2}
    (\int_{\Omega}\frac{|v|^{2}}{|x|^{3}})^{1/2}
    (\int_{\Omega}\frac{|(a\nabla^{b}v)_{T}|^{2}}{|x|})^{1/2}
    +
    C
    \beta_{3}
    \|\nabla^{b}v\|_{\dot Y}
    \|v\|_{\dot Y}
  \end{equation*}
  where $C=2\|a\|_{L^{\infty}}\|\nabla \psi\|_{L^{\infty}}$
  and
  \begin{equation}\label{eq:betasa}
    \beta_{1}=
    \||x|^{3/2}\widehat{db}_{I}\|_{\ell^{1}L^{2}L^{\infty}}
    +
    \||x|^{3/2}|a-I|\widehat{db}_{II}\|_{\ell^{1}L^{2}L^{\infty}},
  \end{equation}
  \begin{equation}\label{eq:betasb}
    \beta_{2}=
    \||x|^{2}\widehat{db}_{II}\|_{L^{\infty}},
    \qquad
    \beta_{3}=
    \||x|\widehat{db}_{III}\|_{\ell^{1}L^{\infty}}.
  \end{equation}
\end{lemma}

\begin{proof}
  We split 
  $I_{b}=I_{b_{I}}+I_{b_{II}}+I_{b_{III}}$ with
  $I_{b_{I}}=2\Im [(a\nabla^{b}v)\cdot \widehat{db}_{I}
    \overline{v}]\psi'$
  and so on. Then
  \begin{equation*}
    \textstyle
    \int_{\Omega}|I_{b_{I}}|
    \le
    C
    \|\nabla^{b}v\|_{\dot Y}
    \||x|^{1/2}\widehat{db}_{I}
      v\|_{\ell^{1}L^{2}L^{2}}\le
    C
    \|\nabla^{b}v\|_{\dot Y}
    \|v\|_{\dot X}
    \||x|^{3/2}\widehat{db}_{I}
      \|_{\ell^{1}L^{2}L^{\infty}}
  \end{equation*}
  where $C=2N\|\nabla \psi\|_{L^{\infty}}$, and similarly
  \begin{equation*}
    \textstyle
    \int_{\Omega}|I_{b_{III}}|
    \le
    C
    \|\nabla^{b}v\|_{\dot Y}
    \||x|^{1/2}\widehat{db}_{III}v\|_{\ell^{1}L^{2}L^{2}}\le
    C
    \|\nabla^{b}v\|_{\dot Y}
    \|v\|_{\dot Y}
    \||x|\widehat{db}_{III}\|_{\ell^{1}L^{\infty}L^{\infty}}.
  \end{equation*}
  Next we note that
  $db_{II}$ is antisymmetric , hence
  $(a \widehat{x})\cdot \widehat{db}_{II}=
    (a \widehat{x})\cdot (db_{II}\ a \widehat{x})=0$, and 
  for any $\gamma\in \mathbb{C}$ we can rewrite $I_{b_{II}}$ as
  \begin{equation*}
    I_{b_{II}}=
    2\Im [(a\nabla^{b}v-\gamma \widehat{x}
      +\gamma \widehat{x}-\gamma a \widehat{x})
      \cdot \widehat{db}_{II}\overline{v}]\psi'.
  \end{equation*}
  If we choose $\gamma=\widehat{x}\cdot a \nabla^{b}v$ we obtain
  \begin{equation*}
    I_{b_{II}}=
    2\Im [(a\nabla^{b}v)_{T}\cdot \widehat{db}_{II}\overline{v}]\psi'
    +
    2\Im [(\widehat{x}\cdot a \nabla^{b}v)
    ((I-a)\widehat{x})\cdot \widehat{db}_{II}\overline{v}]\psi'=:
    I'_{b_{II}}+I''_{b_{II}}.
  \end{equation*}
  We estimate $I''_{b_{II}}$ like $I_{b_{I}}$:
  \begin{equation*}
    \textstyle
    \int_{\Omega}|I''_{b_{II}}|
    \le
    C
    \|\nabla^{b}v\|_{\dot Y}
    \|v\|_{\dot X}
    \||x|^{3/2}|a-I|\widehat{db}_{II}\|_{\ell^{1}L^{2}L^{\infty}}.
  \end{equation*}
  Finally we have
  \begin{equation*}
    \textstyle
    \int_{\Omega}|I'_{b_{II}}|
    \le
    C
    \||x|^{-1/2}(a\nabla^{b}v)_{T}\|_{L^{2}}
    \||x|^{-3/2}v\|_{L^{2}}
    \||x|^{2}\widehat{db}_{II}\|_{L^{\infty}}.
  \end{equation*}
\end{proof}

\subsection{Choice of the weights and main terms}
\label{sub:weights}

We choose, for arbitrary $R>0$,
\begin{equation}\label{eq:ourpsi}
  \psi=
  \frac{1}{2R}|x|^{2}\one{|x|\le R}+|x|\one{|x|>R},
  \qquad
  \phi=-\frac{\widehat{a}}{R}\one{|x|\le R}.
\end{equation}
Note that $\phi$ is not radial. We have then
\begin{equation}\label{eq:Apsifi}
  \psi'=\frac{|x|}{|x|\vee R},
  \qquad
  \psi''=
  \frac1R\one{|x|\le R},
  \qquad
  A \psi+\phi=\frac{\bar{a}-\widehat{a}+|x|\widetilde{a}}{|x|\vee R}.
\end{equation}
since
$A \psi= \widehat{a}\psi''+ \frac{\overline{a}-\widehat{a}}{|x|} \psi'
  + \widetilde{a}\psi'$.
Recalling the notation
\begin{equation*}
  C_{a}(x)=|a(x)-I|+|x||a'(x)|+|x|^{2}|a''(x)|+|x|^{3}|a'''(x)|
\end{equation*}
we have, after a long but easy computation,
\begin{equation}\label{eq:estader}
  |x|^{2}(|A \bar{a}|+|A \widehat{a}|+|x||A\widetilde{a}|+
    |\nabla \widetilde{a}|)+
  |x|(|\nabla \bar{a}|+|\nabla \widehat{a}|+|\widetilde{a}|)
  \le
  C(n,\|C_{a}\|_{L^{\infty}})\cdot C_{a}(x).
\end{equation}
Then for $|x|>R$ we find that
\begin{equation*}
  \textstyle
  A(A \psi+\phi)=
    -\frac{(\bar{a}-\widehat{a})(\bar{a}-3 \widehat{a})}{|x|^{3}}+
    R(x)
\end{equation*}
where
\begin{equation*}
  \textstyle
  R(x)=-\frac{2a(\nabla \bar{a}-\nabla \widehat{a},\widehat{x})
        +\widetilde{a}(\bar{a}-\widehat{a})}{|x|^{2}}+
    \frac{A(\bar{a}-\widehat{a})}{|x|}
    +A \widetilde{a}
\end{equation*}
and by \eqref{eq:estader}
\begin{equation}\label{eq:Rge}
  \textstyle
  |R(x)|
  \le
  C(n,\|C_{a}\|_{L^{\infty}})\cdot \frac{C_{a}(x)}{|x|^{3}},
  \qquad
  |x|>R.
\end{equation}
In the region $|x|<R$ we have instead
\begin{equation*}
  \textstyle
  A(A \psi+\phi)=R(x)=
  \frac{A(\bar{a}-\widehat{a})+\widetilde{a}^{2}+|x|A \widetilde{a}
     +2a(\nabla \widetilde{a},\widehat{x})}{R}
  +
  \frac{\widetilde{a}(\bar{a}-\widehat{a})}{R|x|}
\end{equation*}
and again by \eqref{eq:estader}
\begin{equation}\label{eq:Rle}
  |R(x)|\le
  C(n,\|C_{a}\|_{L^{\infty}})\cdot \frac{C_{a}(x)}{R|x|^{2}},
  \qquad
  |x|\le R.
\end{equation}
Finally, along the sphere $|x|=R$ there is a singularity of delta type,
originated by the term
\begin{equation*}
  \widehat{a}
  \left(
    \frac{\bar{a}-\widehat{a}+|x|\widetilde{a}}{|x|\vee R}
  \right)''
\end{equation*}
and therefore the singular term has the form
\begin{equation*}
  \textstyle
  -\frac{\widehat{a}(\bar{a}-\widehat{a}+R \widetilde{a})}{R^{2}}
  \delta_{|x|=R}.
\end{equation*}
Summing up we have
\begin{equation}\label{eq:a2psi}
  \textstyle
  A(A \psi+\phi)=
  -\frac{(\bar{a}-\widehat{a})(\bar{a}-3 \widehat{a})}{|x|^{3}}
  \one{|x|>R}
  -\frac{\widehat{a}(\bar{a}-\widehat{a}+R \widetilde{a})}{R^{2}}
  \delta_{|x|=R}
  +
  R(x)
\end{equation}
where $R(x)$ satisfies \eqref{eq:Rge}, \eqref{eq:Rle}.
Further, we note that
\begin{equation}\label{ProbSol1}
  |\widehat{a}-1|+|x||\widetilde{a}|\le C_{a}(x),
  \qquad
  |\bar{a}-n|\le n C_{a}(x)
\end{equation}
so that
\begin{equation*}
  \widehat{a}(\bar{a}-\widehat{a}+R \widetilde{a})\ge1
\end{equation*}
provided e.g~$\|C_{a}\|_{L^{\infty}}\le1/6$. Moreover we have
\begin{equation*}
  (\bar{a}-\widehat{a})(\bar{a}-3 \widehat{a})\ge
  (n-1)(n-3)-2(n+2)C_{a}
\end{equation*}
and in conclusion we have proved the inequality
\begin{equation}\label{eq:a2psiest}
  \textstyle
  -A(A \psi+\phi)\ge
  \frac{(n-1)(n-3)}{|x|^{3}}
  \one{|x|>R}
  +
  \frac{1}{R^{2}}
  \delta_{|x|=R}
  +
  R_{1}(x)
\end{equation}
where $R_{1}(x)$ satisfies for all $x$
\begin{equation*}
  |R_{1}(x)|\le
  C(n)\cdot \frac{C_{a}(x)}{|x|^{2}(|x| \vee R)}
\end{equation*}
with a constant $C(n)$ depending only on $n$ (polynomially).

\begin{lemma}[$I_v$]\label{lem:Iv}
  Let $c=c_{1}+c_{2}+c_{3}+c_{4}+c_{5}$,
  with $c_{5}$ supported in $|x|\le1$,
  and $\phi,\psi$ as in \eqref{eq:choicephi}.
  If $n\ge4$ we have, for all $\delta\in(0,1)$,
  \begin{equation}\label{eq:new6-n}
  \begin{split}
    \textstyle
    \sup_{R>0}
    &
    \textstyle
    \int_{\Omega} I_{v}
    \ge
    (\mu_{n}
    -\gamma_{1}
    -c(n)
    (\gamma_{2}+\gamma_{5}+\|C_{a}\|_{L^{\infty}}))
    \||x|^{-3/2}v\|_{L^{2}}^{2}+
    \|v\|_{\dot X}^{2}+
    \\
    +&
    (\lambda-\Gamma_{3}-c(n)\delta^{-1} \Gamma_{4})
    \|v\|_{\dot Y}^{2}
    -(\gamma_{2}+\delta)\|\nabla^{b}v\|_{\dot Y}^{2}
    -\gamma_{5}\||x|^{-1/2}\nabla^{b}v\|_{L^{2}(|x|\le1)}
  \end{split}
  \end{equation}
  where $\mu_{n}=(n-1)(n-3)/2$ and
  ($\partial_{r}:=\widehat{x}\cdot \nabla$)
  \begin{equation*}
    \gamma_{1}=
    \||x|^{2}([\partial_{r}(|x|c_{1})]_{+}+c_{1,-}+
      |a-I|(|x||\nabla c_{1}|+|c_{1}|)\|_{L^{\infty}},
  \end{equation*}
  \begin{equation*}
    \gamma_{2}=
    \||x|^{2}c_{2}\|_{\ell^{2}L^{\infty}}  ,
    \qquad
    \gamma_{5}=\||x|^{2}c_{5}\|_{L^{\infty}},
  \end{equation*}
  \begin{equation*}
    \Gamma_{3}=
    \|[\partial_{r}(|x|c_{3})]_{+}+c_{3,-}+
      |a-I|(|x||\nabla c_{3}|+|c_{3}|\|_{\ell^{2}L^{\infty}},
    \quad
    \Gamma_{4}=
    \|c_{4,-}\|_{\ell^{1}L^{\infty}}+\|c_{4}\|_{\ell^{1}L^{\infty}}^{2}.
  \end{equation*}
  In dimension $n=3$, provided $c_{5}=0$, we have instead
  \begin{equation}\label{eq:new6-3}
    \begin{split}
      \textstyle
      \sup_{R>0}
      \int_{\Omega} I_{v}
      \ge
      (1
      -\gamma_{1}
      &
      -c(n)\gamma_{2}
      -c(n)\|C_{a}\|_{L^{\infty}})
      \|v\|_{\dot X}^{2}
      \\
      +&
      (\lambda-\Gamma_{3}-c(n)\Gamma_{4})
      \|v\|_{\dot Y}^{2}
      -(\gamma_{2}+\delta)\|\nabla^{b}v\|_{\dot Y}^{2}
    \end{split}
  \end{equation}
  where the definition of $\Gamma_{3}$, $\Gamma_{4}$ is the same, while
  \begin{equation*}
    \gamma_{1}=
    \||x|([\partial_{r}(|x|c_{1})]_{+}+c_{1,-}+
      |a-I|(|x||\nabla c_{1}|+|c_{1}|)\|_{\ell^{1}L^{1}L^{\infty}},
  \end{equation*}
  \begin{equation*}
    \gamma_{2}=
    \||x|^{3/2}c_{2}\|_{\ell^{1}L^{2}L^{\infty}}.
  \end{equation*}

\end{lemma}

\begin{proof}
  Integrating $I_{v}$ on $\Omega$ and using
  \eqref{eq:a2psiest} we obtain
  \begin{equation}\label{eq:tem00}
    \textstyle
    \int_{\Omega} I_{v}
    \ge
    \int_{\Omega_{\ge R}}\!\!
    \frac{\mu_{n}|v|^{2}}{|x|^{3}}
    +
    \int_{\Omega_{=R}}\!\!
    \frac{|v|^{2}}{R^{2}}dS
    -
    \int_{\Omega}
    \frac{c(n)C_{a}(x)|v|^{2}}{|x|^{2}(|x|\vee R)}
    +
    \int_{\Omega_{\le R}}\!\!
    \frac{\lambda|v|^{2}}R
    -
    \int_{\Omega}
    (\frac{|x|(a\widehat{x})\cdot \nabla c}{|x|\vee R}
    + \frac{\widehat{a}c}{R}\one{|x|\le R})|v|^{2}.
  \end{equation}
  Consider first the case $n\ge4$.
  We estimate the term
  \begin{equation*}
    \textstyle
    I_{c}=
    (\frac{|x|(a\widehat{x})\cdot \nabla c}{|x|\vee R}
    + \frac{\widehat{a}c}{R}\one{|x|\le R})|v|^{2}
  \end{equation*}
  in two different ways for $c_{1}$, $c_{3}$ and for $c_{2}$, $c_{4}$.
  For $c_{1}$, writing $r=|x|$ and 
  $\partial_{r}=\widehat{x}\cdot \nabla$, we have
  \begin{equation*}
  \begin{split}
    \textstyle
    \frac{|x|(a\widehat{x})\cdot \nabla c_{1}}{|x|\vee R}
        + \frac{\widehat{a}c_{1}}{R}\one{|x|\le R}
    =&
    \textstyle
    \partial_{r}
    \left(\frac{|x|c_{1}}{|x|\vee R}\right)
    +
    (a-I)\widehat{x}\cdot \nabla \left(\frac{|x|c_{1}}{|x| \vee R}\right)
    \\
    \le&
    \textstyle
    \frac1{|x|} 
    \Bigl(
    [\partial_{r}(rc_{1})]_{+}+c_{1,-}
    +|a-I|(|x||\nabla c_{1}|+|c_{1}|)
    \Bigr)
  \end{split}
  \end{equation*}
  so that
  \begin{equation*}
  \begin{split}
    \textstyle
    \sup_{R>0}
    \int_{\Omega}I_{c_{1}}
    \le
    &
    \||x|^{-1}\bigl(
    [\partial_{r}(rc_{1})]_{+}+c_{1,-}
    +
    |a-I|(|x||\nabla c_{1}|+|c_{1}|)
    \bigr)|v|^{2}\|_{L^{1}}
    \\
    \le&
    \gamma_{1} \||x|^{-3/2}v\|_{L^{2}}^{2}.
  \end{split}
  \end{equation*}
  A similar computation for $I_{c_{3}}$ gives
  (also in the case $n=3$)
  \begin{equation*}
    \textstyle
    \sup_{R>0}
    \int_{\Omega}I_{c_{3}}
    \le
    \Gamma_{3}\|v\|_{\dot Y}^{2}.
  \end{equation*}
  On the other hand for $c_{2}$ we write
  \begin{equation*}
    \textstyle
    I_{c_{2}}
    =
    \nabla \cdot
    \left\{\frac{a \widehat{x}|x|c_{2}|v|^{2}}{|x|\vee R}\right\}
    -
    \frac{\bar{a}-\widehat{a}+|x|\widetilde{a}}{|x|\vee R}c_{2}|v|^{2}
    -
    2 \frac{|x|c_{2}}{|x|\vee R}\Re a(\nabla^{b}v,\widehat{x}v)
  \end{equation*}
  and if e.g.~$\|C_{a}\|_{L^{\infty}}\le 1/4$, recalling also
  \eqref{AlsoRec}, we get 
  \begin{equation*}
    \textstyle
    I_{c_{2}}
    \le
    \nabla \cdot
    \left\{\frac{a \widehat{x}|x|c_{2}|v|^{2}}{|x|\vee R}\right\}
    + 
    c(n)
    \frac{c_{2,-}}{|x|}|v|^{2}
   + 4 |c_{2}| |v||\nabla^{b}v|
  \end{equation*}
  so that
  \begin{equation*}
  \begin{split}
    \textstyle
    \sup_{R>0}
    \int_{\Omega} I_{c_{2}}
    \le &
    c(n)\||x|^{-1/2}(c_{2,-})^{1/2}v\|_{L^{2}}^{2}+
    4\|  c_{2}  v  |\nabla^{b}v| \|_{L^1}  
    \\
    \le &
    c(n)\||x|^{2}c_{2,-}\|_{L^{\infty}}\||x|^{-3/2}v\|_{L^{2}}^{2}
    +
    4\||x|^{2}c_{2}\|_{\ell^{2}L^{\infty}}
     \||x|^{-3/2}v\|_{L^{2}}      \|\nabla^{b}v\|_{\dot Y}
    \\
    \le &
    c(n)\gamma_{2}\||x|^{-3/2}v\|_{L^{2}}^{2}+
    \gamma_{2}\|\nabla^{b}v\|_{\dot Y}^{2}.
  \end{split}
  \end{equation*}
  Using the same identity for $c_{4}$ we can estimate
  \begin{equation*}
    \textstyle
    \sup_{R>0}
    \int_{\Omega} I_{c_{4}}
    \le 
    c(n)
    \|c_{4,-}\|_{\ell^{1}L^{\infty}}
    \|v\|_{\dot Y}^{2}
    +
    4
    \|c_{4}\|_{\ell^{1}L^{\infty}}
    \|v\|_{\dot Y}\|\nabla^{b}v\|_{\dot Y}
  \end{equation*}
  and this implies
  \begin{equation*}
    \textstyle
    \sup_{R>0}
    \int_{\Omega} I_{c_{4}}
    \le 
    \delta\|\nabla^{b}v\|_{\dot Y}^{2}
    +
    c(n)\delta^{-1}
    \Gamma_{4}
    \|v\|_{\dot Y}^{2}.
  \end{equation*}
  The same identity for $c_{5}$
  can be estimated as follows, with $C=c(n)$:
  \begin{equation*}
    \textstyle
    \sup_{R>0}
    \int_{\Omega} I_{c_{5}}
    \le
    C
    \gamma_{5}
    \|\frac{v}{{|x|}^{3/2}}\|_{L^{2}}^{2}+
    \gamma_{5}
    \|\frac{\nabla^{b}v}{|x|^{1/2}}\|_{L^{2}(|x|\le1)}^{2}.
  \end{equation*}
  Hence taking the sup in $R>0$ of \eqref{eq:tem00} 
  and using the previous estimates we get \eqref{eq:new6-n}.

  In the case $n=3$ we have $\mu_{3}=0$ and the
  weighted $L^{2}$ norm of $v$ is unavailable.
  We use the $\dot X$ norm instead and we obtain
  \begin{equation*}
    \textstyle
    \sup_{R>0}
    \int_{\Omega}I_{c_{1}}
    \le 
    \gamma_{1}\|v\|_{\dot X}^{2},
  \end{equation*}
  \begin{equation*}
    \textstyle
    \sup_{R>0}
    \int_{\Omega}I_{c_{2}}
    \le
    c(n)\|c_{2,-}|x|\|_{\ell^{1}L^{1}L^{\infty}}\|v\|_{\dot X}^{2}
    +
    4\||x|^{3/2}c_{2}\|_{\ell^{1}L^{2}L^{\infty}}
    \|v\|_{\dot X}\|\nabla^{b}v\|_{\dot Y}
  \end{equation*}
  with the new definition of $\gamma_{1},\gamma_{2}$,
  and this gives \eqref{eq:new6-3}.
\end{proof}

\begin{lemma}[$I_{\nabla v}$]\label{lem:Inablav}
  With $\psi$ as in \eqref{eq:Apsifi}, we have
  \begin{equation}\label{eq:new7}
    \textstyle
    \sup_{R>0}\int_{\Omega}I_{\nabla v}
    \ge
    (1-6\|a-I\|_{L^{\infty}}-
      c(n)\||x|a'\|_{\ell^{1}L^{\infty}})\|\nabla^{b}v\|_{\dot Y}^{2}
    +
    \int_{\Omega}\frac{|(a\nabla^{b}v)_{T}|^{2}}{|x|}.
  \end{equation}
\end{lemma}

\begin{proof}
  By separating the terms in $\alpha_{\ell m}$ which contain
  derivatives of $a_{j k}$ we have
  \begin{equation*}
    I_{\nabla v}=
    s_{\ell m} \cdot 
      \Re(\partial^{b}_{\ell}v\overline{\partial^{b}_{m}v})
    +
    r_{\ell m} \cdot 
      \Re(\partial^{b}_{\ell}v\overline{\partial^{b}_{m}v})
    +
    a(\nabla^{b}v,\nabla^{b}v)\phi
  \end{equation*}
  where
  \begin{equation*}
    \textstyle
    s_{\ell m}(x)=
    2 a_{jm} a_{\ell k} \widehat{x}_{j}\widehat{x}_{k}
    \psi''
    +
    2 [a_{jm}a_{j\ell}-a_{jm} a_{\ell k} \widehat{x}_{j}\widehat{x}_{k}]
    \frac{\psi'}{|x|},
  \end{equation*}
  \begin{equation*}
    r_{\ell m}(x)=
    [2a_{jm}a_{\ell k;j}-a_{jk}a_{\ell m;j}]\widehat{x}_{k}\psi'.
  \end{equation*}
  With our choice of $\psi$ we get
  \begin{equation*}
    \textstyle
    |r_{\ell m}(x)
    \Re(\partial^{b}_{\ell}v\overline{\partial^{b}_{m}v}) |
    \le
    c(n)
    \frac{|x||a'|}{|x|\vee R}
    |\nabla^{b}v|^{2}
    \le
    c(n)|a'||\nabla^{b}v|^{2},
  \end{equation*}
  \begin{equation*}
    \textstyle
    a(\nabla^{b}v,\nabla^{b}v)\phi=
    -\frac{\widehat{a}}{R}
    \one{|x|\le R}
    a(\nabla^{b}v,\nabla^{b}v)
    \ge
    -\frac{N^{2}}{R}
    \one{|x|\le R}
    |\nabla^{b}v|^{2}.
  \end{equation*}
  Moreover
  \begin{equation*}
    \textstyle
    s_{\ell m}
      \Re(\partial^{b}_{\ell}v\overline{\partial^{b}_{m}v})=
    2|(a \nabla^{b}v)_{R}|^{2}\psi''
    +
    2|(a \nabla^{b}v)_{T}|^{2}\frac{\psi'}{|x|}
  \end{equation*}
  which gives, using $|w_{R}|^{2}+|w_{T}|^{2}=|w|^{2}$ (we recall notation \eqref{eq:tangder})
  \begin{equation*}
    s_{\ell m}
      \Re(\partial^{b}_{\ell}v\overline{\partial^{b}_{m}v})
    =
    \textstyle
    \frac2R|a \nabla^{b}v|^{2}
    \one{{|x|\le R}}
    +
    \frac{2}{|x|}
    |(a \nabla^{b}v)_{T}|^{2}
    \one{|x|>R}.
  \end{equation*}
  Summing up we obtain
  \begin{equation*}
    \textstyle
    I_{\nabla v}\ge
    \frac{(2\nu^{2}-N^{2})}R \cdot
    |\nabla^{b}v|^{2}
    \one{|x|\le R}
    +
    \frac{2}{|x|}
    |(a \nabla^{b}v)_{T}|^{2}
    \one{|x|>R}
    -
    c(n)|a'||\nabla^{b}v|^{2}.
  \end{equation*}
  Note that we can assume $\nu\ge1-\|a-I\|_{L^{\infty}}$ and 
  $N\le1+\|a-I\|_{L^{\infty}}$ so that
  \begin{equation*}
    2\nu^{2}-N^{2}\ge 1-6\|a-I\|_{L^{\infty}}.
  \end{equation*}
  Integrating on $\Omega$ and taking the sup over $R>0$ we obtain
  \begin{equation*}
    \textstyle
    \sup_{R>0}\int_{\Omega}I_{\nabla v}
    \ge
    (1-6\|a-I\|_{L^{\infty}})\|\nabla^{b}v\|_{\dot Y}^{2}
    +
    \int_{\Omega}\frac{|(a\nabla^{b}v)_{T}|^{2}}{|x|}
    -
    c(n)\||a'| |\nabla^{b}v|^{2}\|_{L^{1}}
  \end{equation*}
  and this implies the claim.
\end{proof}

\begin{lemma}[$I_{f}$]\label{lem:If}
  With $\phi,\psi$ as in \eqref{eq:ourpsi},
  we have for all $\delta\in(0,1)$
  \begin{equation}\label{eq:new8}
    \textstyle
    \int_{\Omega}I_{f}
    \le
    \delta
    \|v\|_{\dot X}^{2}+
    \delta
    \|\nabla^{b}v\|_{\dot Y}^{2}
    +
    C(n,\|C_{a}\|_{L^{\infty}})\delta^{-1}
    \|f\|_{\dot Y^{*}}^{2}.
  \end{equation}
\end{lemma}

\begin{proof}
  By \eqref{eq:Apsifi}
  \begin{equation*}
    \textstyle
    I_{f}=
    \frac{\bar{a}-\widehat{a}+|x|\widetilde{a}}{|x| \vee R}
    \Re(\overline{v}f)
    +
    \frac{2|x|a \widehat{x}}{|x| \vee R}
    \Re(\overline{\nabla^{b}v}f)
    \le
    C(n,\|C_{a}\|_{L^{\infty}})
    (\frac{|v|}{|x|}+|\nabla^{b}v|)|f|
  \end{equation*}
  and hence
  \begin{equation*}
    \textstyle
    \int_{\Omega}I_{f}
    \le
    C(n,\|C_{a}\|_{L^{\infty}})
    (\||x|^{-1}v\|_{\dot Y}+\|\nabla^{b}v\|_{\dot Y})
    \|f\|_{\dot Y^{*}}.
  \end{equation*}
  The claim follows recalling that 
  $\||x|^{-1}v\|_{\dot Y}\le\|v\|_{\dot X}$.
\end{proof}

\subsection{Conclusion of the proof}\label{sub:cconclusion}

We are ready to complete the proof of Theorem \ref{the:smoo}.
We integrate \eqref{eq:morid} on $\Omega$ with the choice of
weights \eqref{eq:ourpsi} and we take the supremum over $R>0$.
We then apply the previous Lemmas to estimate the individual terms.

We consider first the case (A$_{0}$).
One checks easily that the assumptions on $b,c$
imply the following:
for $b=b_{I}+b_{II}$ we have
\begin{equation*}
  \||x|^{3/2}\widehat{db}_{I}
    \|_{\ell^{1}L^{2}L^{\infty}}+
  \||x|^{3/2}|a-I|\widehat{db}_{II}\|_{\ell^{1}L^{2}L^{\infty}}+
  \||x|^{2}\widehat{db}_{II}\|_{L^{\infty}}
  <\mu,
\end{equation*}
with $b_{II}=0$ in $n=3$,
while the electric potential can be written
$c=c_{1}+c_{2}+c_{f}$ with
\begin{equation*}
  \||x|^{3/2}c_{f}\|_{\ell^{1}L^{2}L^{\infty}}<\mu,\quad c_{1,-}\in L^\infty
\end{equation*}
and in dimension $n\ge4$
\begin{equation*}
  |a-I|\cdot
  (|x|^{2} |c_{1}|+|x|^{3}|\nabla c_{1}|)
  +
  |x|^{2}\cdot(c_{1,-}+[\partial_{r}(rc_{1})]_{+}+c_{2,-})
  +
   \||x|c_{2}\|_{\ell^{1}L^{\infty}}
  <\mu
\end{equation*}
while in dimension $n=3$
\begin{equation*}
  \|
  |a-I|\cdot
  (|x|^{2}|c_{1}|+|x|^{3}|\nabla c_{1}|)
  +
  |x|^{2}\cdot(c_{1,-}+[\partial_{r}(rc_{1})]_{+}+
    c_{2,-})\|_{\ell^{1} L^{\infty}}+
  \||x|c_{2}\|_{\ell^{1} L^{\infty}}
  <\mu.
\end{equation*}
Indeed, it is sufficient to take $c_{1}=c_{I}$
and, for a fixed cutoff $0\le\chi(x)\le1$ supported
near 0, $c_{2}=(1-\chi)\cdot c_{II}$ and $c_{f}=\chi \cdot c_{II}$.

Consider the case $n\ge4$. 
Write $\widetilde{c}=c_{1}+c_{2}$ and
\begin{equation*}
  \widetilde{f}=(A^{b}-\widetilde{c}+\lambda+i \epsilon)v,
  \qquad
  \widetilde{f}=f+c_{f}v.
\end{equation*}
Then all the assumptions of Lemmas 
\ref{lem:1},
\ref{lem:epsilon}, 
\ref{lem:lambdameno},
\ref{lem:magnetic},
\ref{lem:Iv},
\ref{lem:Inablav} and
\ref{lem:If}
are satisfied by $a$, $b$ and $\widetilde{c}$. As a consequence we have
\begin{eqnarray*}
  \textstyle
  \sup\limits_{R>0}
  \int_{\Omega}|I_{\epsilon}|+|I_{b}|+|I_{\widetilde{f}}|
  &\le&
  C\cdot(\delta+\mu)
  (\|v\|_{\dot X}^{2} + \|\nabla^{b}v\|^{2}_{\dot Y}  
  +\||x|^{-3/2}v\|_{L^{2}}^{2}
  )
  \\
  &+&
  [\mu+\delta(|\lambda|+|\epsilon|+\|c_{I,-}\|_{L^\infty})]  \|v\|_{\dot Y}^{2}
  +
  C \delta^{-4}\|\widetilde{f}\|_{\dot Y^{*}}^{2},
\end{eqnarray*}
\begin{equation*}
  \textstyle
  \sup\limits_{R>0}\int_{\Omega}(I_{v}+I_{\nabla v})
  \ge
  (\frac12-C \mu)\||x|^{-3/2}v\|_{L^{2}}^{2}
  +
  (1-C(\mu+\delta))\|\nabla^{b}v\|_{\dot Y}^{2}
  +
  \|v\|_{\dot X}^{2}
  +
  \lambda\|v\|_{\dot Y}^{2},
\end{equation*}
\begin{equation*}
  \lambda_{-}\|v\|_{\dot Y}^{2}
  \le
  C \mu\||x|^{-3/2}v\|_{L^{2}}^{2}
  +
  \delta\|v\|_{\dot X}^{2}
  +
  C \delta^{-1}\|\widetilde{f}\|_{\dot Y^{*}}^{2}.
\end{equation*}
Thus integrating \eqref{eq:morid} on $\Omega$
and dropping the boundary terms, which give a negative contribution
as proved in Remark \ref{rem:bdry}, from the previous
inequalities taking $\delta$ and $\mu$ sufficiently small
we obtain \eqref{eq:newsmoo}, with
$\widetilde{f}=f+c_{f}v$ in place of $f$. 
More precisely, we use Lemma \ref{lem:epsilon} to get rid of the $\epsilon$ term at the right hand side, so that we obtain \eqref{eq:newsmoo} with $\epsilon=0$. To reinclude the $\epsilon$ term, we can use again \eqref{eq:new3} combined with the local smoothing just obtained, which gives
$$
   |\epsilon|^{1/2}\|v\|_{\dot Y}
    \le
    C\|a\|_{L^{\infty}}
    \left(\|\nabla^{b} v\|_{\dot Y}+\|v\|_{\dot X}
    +\|f\|_{\dot Y^{*}}\right)\leq C(a)(1+c(n))\|f\|_{\dot{Y}^*}.
$$
Now it remains  to estimate
\begin{equation*}
  \|f+c_{f}v\|_{\dot Y^{*}}
  \le
  \|f\|_{\dot Y^{*}}+
  \||x|^{3/2}c_{f}\|_{\ell^{1}L^{2}L^{\infty}}
  \|v\|_{\dot X}
  <
  \|f\|_{\dot Y^{*}}+
  \mu\|v\|_{\dot X}
\end{equation*}
and absorb the last term at the left hand side, provided
$\mu$ is small enough.
The proof for $n=3$ is completely analogous.

In the case of the weaker condition (A)
the argument is almost the same. 
We split $c=c_{1}+c_{2}+c_{3}+c_{4}+c_{f}$
with $c_{1}=c_{I}$, $c_{2}=(1-\chi) c_{II}$, $c_{3}=c_{III}$,
$c_{4}=(1-\chi)c_{IV}$ and $c_{f}=\chi\cdot(c_{II}+c_{IV})$,
and we write $\widetilde{c}=c_{1}+c_{2}+c_{3}+c_{4}$ and
\begin{equation*}
  \widetilde{f}=(A^{b}-\widetilde{c}+\lambda+i \epsilon)v,
  \qquad
  \widetilde{f}=f+c_{f}v
\end{equation*}
as before. Note that in the
estimate of $I_{\epsilon}$ we get an additional term
$\|c_{3,-}+c_{4,-}\|_{L^{\infty}}\|v\|_{\dot Y}^{2}$, while 
in estimate \eqref{eq:new6-n} we must take
$\frac12\lambda>c(n)(Z+Z^{2})\ge\Gamma_{3}+c(n)\Gamma_{4}$ in order 
to obtain positive terms. Then we can apply the lemmas as above, and
in the final step we estimate $\widetilde{f}$ as follows:
\begin{equation*}
  \|\widetilde{f}\|_{\dot Y^{*}}
  \le
  \|f\|_{\dot Y^{*}}+
  \||x|^{3/2}\chi c_{II}\|_{\ell^{1}L^{2}L^{\infty}}
  \|v\|_{\dot X}+
  \||x|\chi c_{IV}\|_{\ell^{1}L^{\infty}}\|v\|_{\dot Y}
  \le
  \|f\|_{\dot Y^{*}}+
  \mu
  \|v\|_{\dot X}+
  Z
  \|v\|_{\dot Y}.
\end{equation*}
In conclusion, we arrive at an estimate of the form
\begin{equation*}
  \|v\|_{\dot X}
  +
  (\lambda+|\epsilon|)^{1/2}\|v\|_{\dot Y}
  +
  \|\nabla^{b}v\|_{\dot Y}+
  \|(a \nabla^{b}v)_{T}\|_{L^{2}}
  \le
  c(n)
  \|f\|_{\dot Y^{*}}
  +
  c(n)Z\|v\|_{\dot Y}
\end{equation*}
and the additional term $\|v\|_{\dot Y}$ can be absorbed
at the left hand side, provided $\lambda$ is large enough.
We omit the details.

\begin{remark}[Inverse square potentials]\label{rem:invsq}
  Note that in dimension $n\ge4$ and for $\lambda>0$
  we can add to the electric potential $c$
  a further term $c_{V}$ satisfying
  \begin{equation*}
    \gamma_{5}:=\||x|^{2}c_{V}\|_{L^{\infty}}\ll1
    \qquad
    c_{V} \ \text{supported in}\ \{|x|\le1\}.
  \end{equation*}
  Indeed, taking $c_{5}=c_{V}$ in Lemma \ref{lem:Iv},
  we obtain an additional term
  at the right hand side of the estimate:
  \begin{equation*}
  \begin{split}
    \|v\|_{\dot X}
    +
    \||x|^{-3/2}v\|_{L^{2}}
    +
    (\lambda_{+}+|\epsilon|)^{1/2}\|v\|_{\dot Y}
    +
    \|\nabla^{b}v\|_{\dot Y}+
    \|(a \nabla^{b}v)_{T}\|_{L^{2}}
    \le
    \\
    \le
    c(n)(
    \|f\|_{\dot Y^{*}}
    +
    \gamma_{5}^{1/2}\||x|^{-1/2}\nabla^{b}v\|_{L^{2}(|x|\le1)}).
  \end{split}
  \end{equation*}
  We can estimate the additional term using
  Lemma \ref{lem:invsq}:
  \begin{equation*}
    \||x|^{-1/2}\nabla^{b}v\|_{L^{2}(|x|\le1)}
    \le 
    2 \lambda_{+}\||x|^{-3/2}v\|_{L^{2}(|x|\le2)}^{2}
    +
    c(n)\|v\|_{\dot X}^{2}+c(n)\|f\|_{\dot Y^{*}}^{2}
  \end{equation*}
  and if $\mu$ is small enough we can absorb the $\|v\|_{\dot X}$
  term at the right hand side:
  \begin{equation*}
  \begin{split}
    \|v\|_{\dot X}
    +
    \||x|^{-3/2}v\|_{L^{2}}
    +
    (|\lambda|+|\epsilon|)^{1/2}\|v\|_{\dot Y}
    +
    \|\nabla^{b}v\|_{\dot Y}+
    \|(a \nabla^{b}v)_{T}\|_{L^{2}}
    \le
    \\
    \le
    c(n)
    \|f\|_{\dot Y^{*}}
    +
    c(n)(\gamma_{5}\lambda_{+})^{1/2}\||x|^{-3/2}v\|_{L^{2}}.
  \end{split}
  \end{equation*}
  In conclusion, if we assume
  \begin{equation}\label{eq:smallinvsq}
    \||x|^{2}c_{V}\|_{L^{\infty}} \cdot \lambda_{+}<\epsilon(n)
  \end{equation}
  for a suitable constant $\epsilon(n)$ depending only on $n$,
  we can absorb also the remaining term and we obtain 
  that the estimate \eqref{eq:newsmoo} continues to hold.
  However in this case the condition on $c_{V}$ is not independent
  of $\lambda$ and actually becomes worse as $\lambda_{+}$ grows.
\end{remark}

\section{The radiation estimate} \label{sec:radest}

The goal of this Section is to prove an estimate for the
difference
\begin{equation*}
  \nabla^{b}_{S}v:=\nabla^{b}v-i \widehat{x} \sqrt{\lambda}v
\end{equation*}
($S$ stands for Sommerfeld)
in a norm slightly stronger than $\|\cdot\|_{\dot Y}$; 
to this purpose we use the weighted $L^{2}$ norms,
for some $\delta>0$,
\begin{equation}\label{eq:weightn}
  \int_{\Omega}|x|^{\delta-1}|\nabla^{b}_{S}v|^{2}dx.
\end{equation}
This is enough to rule out functions in the kernel of
$L+\lambda$ and hence to get uniqueness for the
Helmholtz equation. Indeed, if the previous norm is finite
then condition \eqref{eq:somm1} is satisfied.
The value of $\delta$ is connected to the
asymptotic behaviour of the metric $a(x)$ (see the
statement of Theorem \ref{the:sommlarge}), a fact already
observed in \cite{rodtao}.

Note that we can only estimate \eqref{eq:weightn} in
terms of the $\dot Y$ norms of $v$ and its derivative;
in order to get an actual
estimate, this result must be combined
with the smoothing estimate of Section \ref{sec:smooest}.

Since we are interested in the behaviour
of solutions in the limit $\lambda+i \epsilon\to \lambda>0$, it
is actually sufficient to prove an estimate in the
quarter plane $|\epsilon|<\lambda$. However, the estimate in
the case $\lambda\le|\epsilon|$ is elementary (and actually
stronger), and we give it here for the sake of
completeness.

\begin{proposition}[Radiation estimate, Case $\lambda\le|\epsilon|$]
  \label{pro:lambdasma}
  Let $\epsilon\in \mathbb{R}$, $0<\lambda\le|\epsilon|$.
  Assume $C_{a}<1/2$ and $c=c_{I}+c_{II}$ with
  \begin{equation*}
    \||x|^{2}c_{I,-}\|_{L^{\infty}}\le \kappa,
    \qquad
    \|c_{II,-}\|_{L^{\infty}}\le K.
  \end{equation*}
  If $\kappa$
  is sufficiently small with respect to $n$, we have
  \begin{equation}\label{eq:estlambdaep}
    \|\nabla^{b}v\|_{L^{2}}^{2}
    +
    \lambda\|v\|_{L^{2}}^{2}
    \lesssim
    (1+K \lambda^{-1})
    \left[
    \|v\|_{\dot Y}^{2}+
    \|f\|_{\dot Y^{*}}^{2}
    \right].
  \end{equation}
\end{proposition}

\begin{proof}
  We can assume $\epsilon>0$.
  By $\lambda\le \epsilon$ and \eqref{eq:new1} we have
  \begin{equation*}
    \textstyle
    \lambda\int|v|^{2}\le
    \epsilon\int|v|^{2}\le
    \int|f \overline{v}|
    \le\|v\|_{\dot Y}^{2}+
    \|f\|_{\dot Y^{*}}^{2}.
  \end{equation*}
  Also by \eqref{eq:new1}, we can write 
  for all $\delta>0$
  \begin{equation*}
    \textstyle
    \int a(\nabla^{b}v,\nabla^{b}v)
    \le
    \lambda\int|v|^{2}
    +
    \int c_{-}|v|^{2}
    +
    \int|f \overline{v}|
    \le
    (\lambda+K)
    \int|v|^{2}
    +
    \kappa\int \frac{|v|^{2}}{|x|^{2}}
    +
    \int|f \overline{v}|.
  \end{equation*}
  By the magnetic Hardy inequality \eqref{eq:maghardy}
  and the previous inequality we have then
  \begin{equation*}
    \textstyle
    \le
    (2+K \lambda^{-1})
    \int|f \overline{v}|
    +
    \kappa c(n)\|\nabla^{b}v\|_{L^{2}}^{2}
  \end{equation*}
  and if $\kappa$ is sufficiently small we deduce
  \begin{equation*}
    \textstyle
    \|\nabla^{b}v\|_{L^{2}}^{2}
    \lesssim
    (1+K \lambda^{-1})
    \int|f \overline{v}|.
  \end{equation*}
  Appying the Cauchy-Schwartz inequality we obtain
  \eqref{eq:estlambdaep}.
\end{proof}

\begin{theorem}[Radiation estimate, Case $\lambda>|\epsilon|$]
  \label{the:sommlarge}
  Let $\delta\in(0,1]$, $b=b_{I}+b_{II}$, $c=c_{I}+c_{II}$ 
  and assume that $|x|^{3}\nabla c_{I}\in L^{\infty}$ and
  for some constants $\kappa,K$
  \begin{equation*}
    \|C_{a}\|_{L^{\infty}}
    +
    \||x|^{2}\widehat{db}_{I}
      \|_{L^{\infty}}
    +
    \||x|^{2}(\partial_{r}(|x|c_{I}))_{+}\|
      _{L^{\infty}}
    +
    \||x|^{2}c_{I,-}\|_{L^{\infty}}
    \le \kappa,
  \end{equation*}
  \begin{equation*}
    \| |x|^{\delta}(|a-I|+|x| |a'|)\|_{\ell^{1}L^{\infty}}
    +
    \| |x|^{\delta+1}\widehat{db}_{II}\|_{\ell^{1}L^{\infty}}
    +
    \||x|^{\delta}c_{II}\| _{\ell^{1}L^{\infty}}
    \le K.
  \end{equation*}
  If $\kappa$ is sufficiently small with respect to $n,\delta$,
  then we have for $\delta<1$
  \begin{equation}\label{eq:sommerest1b}
  \begin{split}
    \textstyle
    (1-\delta)
    &
    \||x|^{\delta-1} 
    (a \nabla^{b}v)_{T}
    \|_{L^{2}}^{2}
    \textstyle
    +\int (|x|^{\delta-1}+\frac{\epsilon}{\sqrt{\lambda}}|x|^{\delta})
    |\nabla_{S}^{b}v|^{2}
    \lesssim
    \\
    \lesssim
    &
    \textstyle
    (1+K)
    \left[
    (1+\lambda)\|v\|_{\dot Y}^{2}+\|\nabla^{b} v\|_{\dot Y}^{2}
      + \int|x|^{\delta}\bra{x}|f|^{2}
      \right] +  K \lambda^{-1} \int|x|^{\delta}\bra{x}|f|^{2},
  \end{split}
  \end{equation}
  while for $\delta=1$ we have
  \begin{equation}\label{eq:sommerest2b}
    \textstyle
    \int (1+\frac{\epsilon}{\sqrt{\lambda}}|x|)
    |\nabla^{b}_{S}v|^{2}
    \lesssim
    (1+K)
    \left[
    (1+\lambda)\|v\|_{\dot Y}^{2}+\|\nabla^{b} v\|_{\dot Y}^{2}
      + \lambda^{-1}  \|f\|_{\dot Y^{*}}^{2}+
      \int|x|^{2}|f|^{2}
    \right].
  \end{equation}
 \end{theorem}

\begin{proof}
  In the proof we shall use the shorthand notation
  \begin{equation*}
    a(w):=a(w,w)=a(x)w \cdot \overline{w},
    \qquad
    w\in \mathbb{C}^{n}
  \end{equation*}
  for the quadratic form associated to the matrix $a$.
  We can assume $\epsilon\ge0$, the other case being similar.

  For later use we write the computations in terms of
  a generic weight function $\chi$ as far as possible.
  We consider again identity \eqref{eq:morid} with the choices
  \begin{equation*}
    \textstyle
    \psi'=\chi
    \quad\text{i.e.}\quad
    \psi(|x|)=\int_{0}^{|x|}\chi(s)ds, 
    \qquad
    \phi=-\chi'+\frac{\epsilon}{\sqrt{\lambda}}\chi
  \end{equation*}
  where $\chi$ is a smooth radial function with
  $\chi,\chi'\ge0$,
  and we add to it the imaginary part of identity \eqref{eq:morid2}
  with the choice $\phi=-2 \sqrt{\lambda}\chi$.
  We also rearrange the terms using the identities
  \begin{equation*}
    I_{\epsilon}=2 \epsilon\Im[a(\nabla \psi,\nabla^{b}v)v]
    =
    \left[
    a(\nabla^{b}v-i \widehat{x} \sqrt{\lambda}v)-
    a(\nabla^{b}v)-\widehat{a}\lambda|v|^{2}
    \right]
    \frac{\epsilon}{\sqrt{\lambda}}\chi
  \end{equation*}
  and
  \begin{equation}
    \Im a(\nabla^{b} v,v \nabla (-2 \sqrt{\lambda}\chi))
    =
    \left[
    a(\nabla^{b}v-i \widehat{x} \sqrt{\lambda}v)-
    a(\nabla^{b}v)-\widehat{a}\lambda|v|^{2}
    \right]
    \chi'.
  \end{equation}
  We obtain the following identity:
  \begin{equation}\label{eq:morid3}
    I_{S}+I_{\nabla v}+I_{v}+I_{c}+I_{b}+I_{f}=
    \partial_{j}\{\Re Q_{j}+\Re P_{j}+\Im \widetilde{P}_{j}\}
  \end{equation}
  where
  \begin{equation*}
    \textstyle
    I_{S}=
    \left[
      \chi'+\frac{\epsilon}{\sqrt{\lambda}}\chi
    \right]
    a(\nabla^{b}v-i \sqrt{\lambda} \widehat{x}v)
  \end{equation*}
  \begin{equation*}
    \textstyle
    I_{\nabla v}=
    2
    |(a \nabla^{b}v)_{R}|^{2}\chi'
    +
    2
    |(a \nabla^{b}v)_{T}|^{2}\frac{\chi}{|x|}
    -
    2
    a(\nabla^{b}v)\chi'
    +
    r_{\ell m} 
      \Re(\partial^{b}_{\ell}v\overline{\partial^{b}_{m}v})
  \end{equation*}
  with
  $r_{\ell m}(x)= 
  [2a_{jm}a_{\ell k;j}-a_{jk}a_{\ell m;j}]\widehat{x}_{k}\chi$
  and using notation \eqref{eq:tangder},
  \begin{equation*}
    \textstyle
    I_{v}=
    \left[
      -\frac12 A(A \psi+\phi)
      +(1-\widehat{a})(\epsilon \sqrt{\lambda}\chi+\lambda \chi')
    \right]
    |v|^{2}
  \end{equation*}
  \begin{equation*}
    \textstyle
    I_{c}=
    \left[
      \frac{\epsilon}{\sqrt{\lambda}}\chi c-
      \chi'c
      -a(\widehat{x},\nabla c)\chi
    \right]
    |v|^{2}
  \end{equation*}
  \begin{equation*}
    I_{b}=
    2\Im [(a\nabla^{b}v)\cdot (db)_{T})\overline{v}]\chi
  \end{equation*}
  \begin{equation*}
    I_{f}=
    \Re [(A\psi+\phi)\overline{v}f
    +2a(\widehat{x},\nabla^bv)\chi f]
    -2 \sqrt{\lambda}\Im(\overline{v}f \chi)
  \end{equation*}
  where
  \begin{equation*}
    Q_{j}=
    a_{jk}\partial^{b}_{k}v \cdot 
      [A^{b},\psi]\overline{v}
      -\textstyle\frac12 a_{jk}(\partial_{k}A \psi)|v|^{2}
      -a_{jk}\widehat{x}_{k}\chi 
      \left[(c-\lambda)|v|^{2}+a(\nabla^{b}v)\right]
  \end{equation*}
  with $\psi'=\chi$, and
  \begin{equation*}
    \textstyle
    P_{j}=
    a_{jk}\partial^{b}_{k}v \overline{v}
    [\frac{\epsilon}{\sqrt{\lambda}}\chi-\chi']
   -\frac12 a_{jk}\widehat{x}_{k}|v|^{2}
    [\frac{\epsilon}{\sqrt{\lambda}}\chi'-\chi''],
    \qquad
    \widetilde{P}_{j}=
    a_{jk}\overline{v}\partial^{b}_{k}v\chi.
  \end{equation*}
  Note that at $\partial \Omega$ the 
  boundary terms $P_{j},\widetilde{P}_{j}$ vanish,
  while $Q_{j}$ give a negative contribution
  as proved in Remark \ref{rem:bdry}; on the other
  hand, the integrals of $P_{j},\widetilde{P}_{j},Q_{j}$
  on the sphere $\{|x|=M\}$ tend to zero as $M\to \infty$
  by the conditions imposed on the growth of $\chi$.
  Hence by integrating \eqref{eq:morid3}
  on $\Omega\cap\{|x|\le M\}$ and letting
  $M\to \infty$ we can neglect the boundary terms
  and we obtain
  \begin{equation*}
    \textstyle
    \int_{\Omega}(I_{S}+I_{\nabla v}+I_{v}+I_{c}+I_{b}+I_{f})\le0.
  \end{equation*}
  We shall also use the 
  magnetic Hardy inequality \eqref{eq:maghardy} for
  different choices of $s$.
  Note that with the substitution
  $w=e^{-i \sqrt{\lambda}|x|}v$ we have also
  \begin{equation}\label{eq:magharS}
    \textstyle
    \||x|^{-s}v\|_{L^{2}}
    \le
    \frac{2}{n-2s}
    \||x|^{1-s}\nabla^{b}_{S}v\|_{L^{2}}
  \end{equation}
  where we used the notation 
  $\nabla^{b}_{S}=\nabla^{b}-i \widehat{x} \sqrt{\lambda}$.

  We estimate each term separately. We can write
  \begin{equation*}
    \textstyle
    I_{\nabla v}=
    2 \chi' a(\nabla^{b}v,(a-I)\nabla^{b}v)
    +
    2(\frac{\chi}{|x|}-\chi')
    |(a \nabla^{b}v)_{T}|^{2}
    +  r_{\ell m} 
      \Re(\partial^{b}_{\ell}v\overline{\partial^{b}_{m}v})
  \end{equation*}
  and noticing that $\chi\ge|x|\chi'$ 
  for $\chi=|x|^{\delta}$, $\delta\le1$,
  we obtain
  \begin{equation}\label{eq:prima}
    \textstyle
    \int_{\Omega}
    I_{\nabla v}
    \ge
    -
    c(n)
    \||x|^{\delta}(|a-I|+|x||a'|)\|_{\ell^{1}L^{\infty}}
    \cdot
    \|\nabla^{b}v\|_{\dot Y}^{2}
    +
    (1-\delta)\||x|^{\frac{\delta-1}2}
    (a \nabla^{b}v)_{T}\|_{L^{2}}^{2}.
  \end{equation}

 In order to estimate $I_{v}$ we first compute
  \begin{equation*}
    \textstyle
    A \psi+\phi=
    (\widehat{a}-1)\chi'
    +
    \frac{\bar{a}-\widehat{a}+|x|\widetilde{a}}{|x|}
    \chi
    +
    \frac{\epsilon}{\sqrt{\lambda}}\chi.
  \end{equation*}
  Recalling \eqref{eq:estader} we have easily
  \begin{equation}\label{eq:estapsifi}
    \textstyle
    |A \psi+\phi|
    \le
    c(n)(\frac{\chi}{|x|}+\chi')+
    \frac{\epsilon}{\sqrt{\lambda}}\chi
    \le
    c(n)|x|^{\delta-1}
    +
    \frac{\epsilon}{\sqrt{\lambda}}|x|^{\delta},
  \end{equation}
  while a straightforward computation gives,
  with $\mu_{n}=(n-1)(n-3)$,
  \begin{equation*}
    \textstyle
    A(A \psi+\phi)\le
    -\frac{\mu_{n}}{|x|^{2}}(\frac{\chi}{|x|}-\chi')
    + \frac{n-1}{|x|}\chi''
    +c(n)\frac{C_{a}C_{\chi}}{|x|^{2}}
    + \frac{\epsilon}{\sqrt{\lambda}}
    \left(
    \widehat{a}\chi''+\frac{n-1}{|x|}\chi'
    + c(n)\frac{C_{a}\chi'}{|x|}
    \right)
  \end{equation*}
  where 
  \begin{equation*}
    C_{\chi}(x):=
      |x|^{-1}\chi+
      \chi'+
      |x||\chi''|+
      |x|^{2}|\chi'''|.
  \end{equation*}
  With the choice $\chi=|x|^{\delta}$, and dropping
  a negative term, this reduces to
  \begin{equation*}
    \textstyle
    A(A \psi+\phi)\le
    -
    \frac{(1-\delta)(n-3+\delta)}{|x|^{3-\delta}}
    +\frac{c(n)C_{a}}{|x|^{3-\delta}}
    + \frac{\epsilon \delta}{\sqrt{\lambda}}
    \frac{n-1+C_{a}c(n)}{|x|^{2-\delta}}.
  \end{equation*}
  We shall drop also the first term at the right, although it gives
  a positive contribution, since it can be recovered from the
  final estimate.
  Thus we have
  \begin{equation*}
    \textstyle
    I_{v}
    \ge
    -\frac{c(n)C_{a}}{|x|^{3-\delta}}|v|^{2}
    -\frac{\epsilon \delta}{\sqrt{\lambda}}
    \frac{n-1+C_{a}c(n)}{|x|^{2-\delta}}|v|^{2}
    -
    |a-I|
    (\epsilon \sqrt{\lambda}\chi + \lambda \chi')
    |v|^{2}.
  \end{equation*}
  We now integrate $I_{v}$ on $\Omega$.
  Thanks to the magnetic Hardy inequality \eqref{eq:magharS}
  with $s=(3-\delta)/2$
  and using the previous estimate for $A(A \psi+\phi)$, we have
  \begin{equation*}
    \textstyle
    c(n)
    \int C_{a}|x|^{\delta-3}|v|^{2}
    \le 
    \frac{4c(n)\|C_{a}\|_{L^{\infty}}}{(n-3+\delta)^{2}}
    \int
    |x|^{\delta-1}|\nabla^{b}_{S}v|^{2}
    \le
    \sigma
    \int I_{S}
  \end{equation*}
  (note that in 3D the constant $\to \infty$ as $\delta\to 0$)
  provided
  \begin{equation*}
    \textstyle
    \frac{4c(n)\|C_{a}\|_{L^{\infty}}}{\nu(n-3+\delta)^{2}}\le 
    \sigma \cdot \delta.
  \end{equation*}
  Here $\sigma$ is a universal constant (it will be chosen 
  equal to $1/10$) which we keep around to track the 
  smallness assumptions on the coefficients.
  In a similar way, with $s=(2-\delta)/2$,
  \begin{equation}\label{eq:firstposs}
    \textstyle
    \frac{\epsilon}{\sqrt{\lambda}}\delta
    \int\frac{n-1+C_{a}c(n)}{|x|^{2-\delta}}|v|^{2}
    \le
    \frac{4\delta(n-1+c(n)\|C_{a}\|_{L^{\infty}})}
    {(n-2-\delta)^{2}}
    \int \frac{\epsilon}{\sqrt{\lambda}} 
    |x|^{\delta}|\nabla^{b}_{S}v|^{2}
    \le
    \sigma
    \int I_{S}
  \end{equation}
  provided
  \begin{equation*}
    \textstyle
    \frac{4 (n-1+c(n)\|C_{a}\|_{L^{\infty}})}{\nu(n-2-\delta)^{2}}
    \cdot \delta\le \sigma.
  \end{equation*}
  Note that the last condition restricts $\delta$ to an interval
  $(0,\delta_{n}]$ which covers $(0,1]$ only for $n$ sufficiently
  large. To get around this difficulty we give an alternative
  estimate of the $\epsilon$ term.
  Fix $\alpha>0$ and split the integral in the regions
  $|x|\le \alpha$ and $|x|\ge \alpha$: 
  \begin{equation*}
    \textstyle
    \epsilon
    \int \frac{|v|^{2}}{|x|^{2-\delta}}
    \le
    \epsilon
    \alpha\int \frac{|v|^{2}}{|x|^{3-\delta}}
    +
    \epsilon\alpha^{\delta-2} \int |v|^{2}
    \le
    \frac{4 \epsilon \alpha}{\nu(n-3+\delta)^{2}}
    \int
    I_{S}
    +
    \alpha^{\delta-2} \int |f\overline{v}|
  \end{equation*}
  where we used again \eqref{eq:magharS} and 
  the inequality
  $\epsilon\int_{\Omega}|v|^{2} \le \int_{\Omega}|f\overline{v}|$
  (recall the first identity \eqref{eq:new1}).
  Hence we obtain
  \begin{equation*}
    \textstyle
    \frac{\epsilon}{\sqrt{\lambda}}\delta
    \int\frac{n-1+C_{a}c(n)}{|x|^{2-\delta}}|v|^{2}
    \le
    C_{1}
    \frac{\epsilon \delta \alpha}{\sqrt{\lambda}}
    \int I_{S}
    +
    C_{2}\frac{\delta\alpha^{\delta-2}}{\sqrt{\lambda}}
    \int|f\overline{v}|
  \end{equation*}
  where
  \begin{equation*}
    \textstyle
    C_{1}=\frac{4(n-1+c(n)\|C_{a}\|_{L^{\infty}})}
        {\nu(n-3+\delta)^{2}},
    \qquad
    C_{2}=n-1+c(n)\|C_{a}\|_{L^{\infty}}.
  \end{equation*}
  We choose now
  \begin{equation*}
    \textstyle
    \alpha=\frac{\sigma}{C_{1}\delta \sqrt{\lambda}}
  \end{equation*}
  and we arrive at the following inequality, which is 
  valid for all $\delta\in(0,1]$:
  \begin{equation*}
    \textstyle
    \frac{\epsilon}{\sqrt{\lambda}}\delta
    \int\frac{n-1+C_{a}c(n)}{|x|^{2-\delta}}|v|^{2}
    \le
    \frac{\epsilon}{\lambda}\sigma\int I_{S}+
    C_{3}
    \sqrt{\lambda}^{1-\delta}
    \int|f\overline{v}|
  \end{equation*}
  where 
  \begin{equation*}
    \textstyle
    C_{3}=
    \frac{4^{2-\delta}[\delta(n-1+c(n)\|C_{a}\|_{L^{\infty}})]^{3-\delta}}
    {(\sigma \nu(n-3+\delta)^{2})^{2-\delta}}
  \end{equation*}
  and we can estimate the coefficient $\epsilon/\lambda$ with
  1 since $\lambda\ge \epsilon$.
  Thus we get
  \begin{equation*}
    \textstyle
    \frac{\epsilon}{\sqrt{\lambda}}\delta
    \int\frac{n-1+C_{a}c(n)}{|x|^{2-\delta}}|v|^{2}
    \le
    \sigma\int I_{S}
    +
    c(n,\delta)(1+\lambda)^{1-\delta}\|v\|_{\dot Y}^{2}+
    \|f\|_{\dot Y^{*}}^{2}.
  \end{equation*}
  Moreover we have
  \begin{equation*}
    \textstyle
    \int|a-I|\lambda \chi'|v|^{2}
    \le
    \||a-I| |x|^{\delta}\|_{\ell^{1}L^{\infty}}
    \lambda\|v\|_{\dot Y}^{2}
  \end{equation*}
  \begin{equation*}
    \textstyle
    \int|a-I|\epsilon \sqrt{\lambda}\chi|v|^{2}
    \le
    \||a-I||x|^{\delta}\|_{L^{\infty}}\sqrt{\lambda}\int|f\overline{v}|
  \end{equation*}
  where we used the estimate
  $\epsilon\int_{\Omega}|v|^{2} \le \int_{\Omega}|f\overline{v}|$
  which follows from \eqref{eq:new1}.
  Summing up, we obtain, as $\delta\in(0,1]$,
  \begin{equation}\label{eq:seconda}
    \textstyle
    \int I_{v}
    \ge
    -2 \sigma\int I_{S}
    -
    c(n,\delta)(1+K)
    ((1+\lambda)\|v\|_{\dot Y}^{2}
    +\|f\|_{\dot Y^{*}}^{2}).
  \end{equation}

  The term $I_{b}$ can be estimated as follows.
  We note that
  \begin{equation*}
    a\nabla^{b}v \cdot (db)_{T}=
    a\nabla^{b}_{S}v\cdot (db)_{T}
  \end{equation*}
  so that, with the choice $\chi=|x|^{\delta}$,
  \begin{equation*}
    \textstyle
    \int I_{b_{I}}\ge
    -c(n)\||x|^{\frac{\delta-1}2}\nabla^{b_{I}}_{S}v\|_{L^{2}}
    \||x|^{2}\widehat{db}_{I}
      \|_{L^{\infty}}
    \||x|^{\frac{\delta-3}2}v\|_{L^{2}}
  \end{equation*}
  and using the magnetic Hardy inequality
  \begin{equation*}
    \textstyle
    \int I_{b_{I}}\ge
    -\frac{2c(n)}{(n-3+\delta)}
    \||x|^{2}\widehat{db}_{I}
      \|_{L^{\infty}}
    \||x|^{\delta-1}\nabla^{b}_{S}v\|_{L^{2}}^{2}
    \ge - \sigma\int I_{S}
  \end{equation*}
  provided
  \begin{equation*}
    \textstyle
    \frac{2c(n)}{\nu(n-3+\delta)}
        \||x|^{2}\widehat{db}_{I}
          \|_{L^{\infty}}
    \le
    \frac{2c(n)}{\nu(n-3+\delta)}
    \cdot \kappa\le \sigma \cdot \delta.
  \end{equation*}
  For the second piece $I_{b_{II}}$ we have simply
  \begin{equation*}
    \textstyle
    \int I_{b_{II}}\ge
    -c(n)\|\nabla^{b}v\|_{\dot Y}\|v\|_{\dot Y}
    \||x|^{1+\delta}\widehat{db}_{II}\|_{\ell^{1}L^{\infty}}
    \ge
    -c(n)K\|\nabla^{b}v\|_{\dot Y}\|v\|_{\dot Y}
  \end{equation*}
  and in conclusion
  \begin{equation}\label{eq:terza}
    \textstyle
    \int I_{b}\ge
    - \sigma\int I_{S}
    -c(n)K\|\nabla^{b}v\|_{\dot Y}\|v\|_{\dot Y}.
  \end{equation}



  To estimate $I_{c}$ we begin by writing, with $\chi=|x|^{\delta}$,
  \begin{equation*}
    \textstyle
    \int
    I_{c_{I}}\ge
    -\frac{\epsilon}{\sqrt{\lambda}}
    \||x|^{2}c_{I,-}\|_{L^{\infty}}\int||x|^{\delta-2}|v|^{2}
    -\int[\delta|x|^{\delta-1} c_{I}
    +|x|^{\delta}a(\widehat{x}, \nabla c_{I})]|v|^{2}
  \end{equation*}
  and the first term can be handled again using Hardy's inequality:
  \begin{equation*}
    \textstyle
    \ge
    -\sigma\int I_{S}
    -\int[\delta|x|^{\delta-1} c_{I}
    +|x|^{\delta}a(\widehat{x}, \nabla c_{I})]|v|^{2}
  \end{equation*}
  provided
  \begin{equation*}
    \textstyle
    \frac{4}{\nu(n-2+\delta)^{2}}\||x|^{2}c_{I,-}\|_{L^{\infty}}
    \le
    \frac{4}{\nu(n-2+\delta)^{2}}\cdot\kappa
    \le \sigma \cdot \delta.
  \end{equation*}
  To bound the second integral we write,
  with $\partial_{r}$ denoting the radial derivative,
  \begin{equation*}
    \delta|x|^{\delta-1} c_{I} +
    |x|^{\delta}a(\widehat{x}, \nabla c_{I} )]|v|^{2}
    =
    \partial_{r}(|x|^{\delta}c_{I} )
    +
    (a-I)\widehat{x}\cdot\nabla c_{I} |x|^{\delta}
  \end{equation*}
  \begin{equation*}
    =
    \left((\delta-1)|x|^{2}c_{I}+|x|^{2}\partial_{r}(|x|c_{I})\right)
    |x|^{\delta-3}
    +
    (a-I)\widehat{x}\cdot\nabla c_{I} |x|^{\delta}
  \end{equation*}
  \begin{equation*}
    \le
    \kappa \cdot
    (1+\|x|^{3}\nabla c_{I} \|_{L^{\infty}})
    \cdot
    |x|^{\delta-3}
  \end{equation*}
  and hence, using Hardy's inequality,
  \begin{equation*}
    \textstyle
    \int[\delta|x|^{\delta-1} c_{I} 
    +a(\widehat{x}, \nabla c_{I} )\chi]|v|^{2}
    \le
    \sigma\int I_{S}
  \end{equation*}
  provided
  \begin{equation*}
    \textstyle
    \frac{4}{\nu(n-3+\delta)^{2}}
    (1+\|x|^{3}\nabla c_{I} \|_{L^{\infty}})
    \cdot
    \kappa
    \le \sigma \cdot \delta.
  \end{equation*}
  Thus we have proved, for $\kappa$ small enough,
  \begin{equation*}
    \textstyle
    \int I_{c_{I}}\ge
    -2 \sigma\int I_{S}.
  \end{equation*}
  For the second piece $I_{c_{II}}$ we use again \eqref{eq:new1}:
  with $\chi=|x|^{\delta}$, we have
  \begin{equation*}
    \textstyle
    \int
    I_{c_{II}}\ge
    -\lambda^{-1/2}\||x|^{\delta} c_{II,-}\|_{L^{\infty}}
    \int|f\overline{v}|
    -\int[\chi' c_{II}+a(\widehat{x}, \nabla c_{II})\chi]|v|^{2}.
  \end{equation*}
  Using the identity ($c=c_{II}$)
  \begin{equation*}
    \textstyle
    a(\widehat{x},\nabla c)\chi|v|^{2}=
    \partial_{j}\{a_{j k}\widehat{x}_{k}c \chi|v|^{2}\}
    -
    \frac{\bar{a}-\widehat{a}+|x|\widetilde{a}}{|x|}c \chi|v|^{2}
    -
    \widehat{a}c \chi'|v|^{2}
    -
    2\Re a(\nabla^{b}v,\widehat{x}v) c \chi
  \end{equation*}
  we obtain
  \begin{equation*}
    \textstyle
    \int[\chi' c_{II}+a(\widehat{x}, \nabla c_{II})\chi]|v|^{2}
    \le
    c(n)
    \||x|^{\delta}c_{II}\|_{\ell^{1}L^{\infty}}
    (\|v\|_{\dot Y}^{2}+\|\nabla^{b}v\|_{\dot Y}^{2}).
  \end{equation*}
  Summing up, we have proved
  \begin{equation}\label{eq:quarta}
    \textstyle
    \int I_{c}\ge
    -2 \sigma\int I_{S}
    -\lambda^{-1}K\|f\|_{\dot Y^{*}}^{2}
        -
        c(n)K(\|\nabla^{b}v\|_{\dot Y}^{2}+\|v\|_{\dot Y}^{2}).
  \end{equation}

  Finally for $I_{f}$ we can write
  \begin{equation*}
    2\Re a(\widehat{x},\nabla^bv)\chi f
      -2 \sqrt{\lambda}\Im(\overline{v}f \chi)
    =
    2\Re(a-I)\widehat{x}\cdot \overline{\nabla^{b}v}\chi f
    +
    2\Re\widehat{x}\cdot \overline{\nabla^{b}_{S}v}\chi f
  \end{equation*}
  and recalling \eqref{eq:estapsifi}
  \begin{equation*}
    \textstyle
    I_{f}\ge
    -
    (c(n)|x|^{\delta-1}
    +
    \frac{\epsilon}{\sqrt{\lambda}}|x|^{\delta})
    |f\overline{v}|
    -|a-I| |x|^{\delta} |\nabla^{b}v||f|
    -
    2|x|^{\delta}|\nabla^{b}_{S}v| |f|.
  \end{equation*}
  The integral of the first term is estimated by Cauchy-Schwartz
  \begin{equation*}
    \textstyle
    \int|x|^{\delta-1}|f\overline{v}|
    \le
    \alpha \delta\int|x|^{\delta-3}|v|^{2}+
    \frac{1}{\alpha \delta}\int|x|^{\delta+1}|f|^{2}
  \end{equation*}
  and  then by Hardy's inequality
  \begin{equation*}
    \textstyle
    \alpha \delta\int|x|^{\delta-3}|v|^{2}
    \le
    \frac{4 \alpha \delta}{(n-3+\delta)^{2}}
    \int|x|^{\delta-1}|\nabla_{S}^{b}v|^{2}
    \le \sigma\int I_{S},
    \qquad
    4\alpha=\sigma(n-3+\delta)^{2}\nu
  \end{equation*}
  and we conclude
  \begin{equation*}
    \textstyle
    \int|x|^{\delta-1}|f\overline{v}|
    \le
    \sigma\int I_{S}
    +
    c(n,\delta)\int|x|^{\delta+1}|f|^{2}.
  \end{equation*}
  For the second term we use the condition
  $\epsilon\le \lambda$ and we obtain
  \begin{equation*}
    \textstyle
    \frac{\epsilon}{\sqrt{\lambda}}
    \int
    |x|^{\delta}|f\overline{v}|
    \le
    \epsilon\int|v|^{2}+\int|x|^{2 \delta}|f|^{2}
    \le
    \int|f\overline{v}|+\int|x|^{2 \delta}|f|^{2}
  \end{equation*}
  Next we have
  \begin{equation*}
    \textstyle
    \int
    |a-I| |x|^{\delta} |\nabla^{b}v||f|
    \le
    \||x|^{\delta}(a-I)\|_{L^{\infty}}
    \|\nabla^{b}v\|_{\dot Y}\|f\|_{\dot Y^{*}}.
  \end{equation*}
  The integral of remaining term can be estimated as follows:
  \begin{equation*}
    \textstyle
    \int |x|^{\delta}|\nabla^{b}_{S}v| |f|\le
    \sigma \delta\nu
    \int|x|^{\delta-1}|\nabla_{S}^{b}v|^{2}
    + \frac{1}{\sigma \delta\nu}
      \int|x|^{\delta+1}|f|^{2}
    \le
    \sigma\int I_{S}
    + \frac{1}{\sigma \delta\nu}
      \int|x|^{\delta+1}|f|^{2}.
  \end{equation*}
 Summing up, we have proved
  \begin{equation}\label{eq:quinta}
    \textstyle
    \int_{\Omega}I_{f}
    \ge
    -2\sigma\int I_{S}
    -c\int(|x|^{\delta+1}+|x|^{2 \delta})|f|^{2}
    -\int|f\overline{v}|
    -     
    K
    \|\nabla^{b}v\|_{\dot Y}\|f\|_{\dot Y^{*}}
  \end{equation}
  for some $c=c(n,\sigma,\delta)$.

  We collect \eqref{eq:prima}, \eqref{eq:seconda}, \eqref{eq:terza},
  \eqref{eq:quarta} and \eqref{eq:quinta} to obtain
  \begin{equation*}
  \begin{split}
    \textstyle
    (1-7 \sigma)\int 
    &
    I_{S}+
    \textstyle
    (1-\delta)\||x|^{\frac{\delta-1}{2}}
    (a \nabla^{b}v)_{T}\|_{L^{2}}^{2}
    \le
    \textstyle
    c(n,\delta)
    \int(|x|^{\delta+1}+|x|^{2 \delta})|f|^{2}+
    \\
    &
    \textstyle
    +c(n,\delta)(1+K)
    ((1+\lambda)\|v\|_{\dot Y}^{2}+\|\nabla^{b} v\|_{\dot Y}^{2}
      + \lambda^{-1}  \|f\|_{\dot Y^{*}}^{2}).
  \end{split}
  \end{equation*}
  We now choose $\sigma=1/10$ so that $1-7 \sigma>0$. Moreover,
  in the case $\delta<1$ we have easily
  \begin{equation*}
    \textstyle
    \int(|x|^{\delta+1}+|x|^{2 \delta})|f|^{2}+
    \|f\|_{\dot Y^{*}}^{2}
    \lesssim
    \int|x|^{\delta}\bra{x}|f|^{2}
  \end{equation*}
  and this gives \eqref{eq:sommerest1b},
  while for $\delta=1$
  we leave the two norms of $f$ separate, and we obtain 
  \eqref{eq:sommerest2b}.
\end{proof}

\section{Proof of Theorem \ref{the:existence}} 
\label{sec:conclusionprf}

We first prove that the only solution satisfying the 
Sommerfeld condition is 0.

\begin{corollary}[Uniqueness]\label{cor:uniq}
  Assume (A) holds,
  $\mu<\mu_{0}(n)$ and $\lambda\ge c_{0}(n)(Z+Z^{2})$.
  Let $v\in H^{1}_{loc}(\Omega)$ with $v\vert_{\partial \Omega}=0$
  be a solution of
  \begin{equation*}
    (L+\lambda )v=0
  \end{equation*}
  satisfying the Sommerfeld radiation condition
  \begin{equation}\label{eq:somm}
    \liminf_{R\to \infty}
    \int_{|x|=R}
    \bigl|\nabla^{b}v-i \sqrt{\lambda}\widehat{x}v\bigr|^{2}dS
    =0.
  \end{equation}
  Then $v \equiv0$.
  If in particular
  \begin{equation}\label{eq:deltawei}
    \int_{|x|\gg1}|x|^{\delta-1}\bigl|\nabla^{b}v-i 
    \sqrt{\lambda}\widehat{x}v\bigr|^{2}dx<\infty
  \end{equation}
  for some $\delta>0$, then \eqref{eq:somm} is satisfied and the
  same conclusion holds.
\end{corollary}

\begin{proof}[Proof of the Corollary]
  By the assumptions on $L$ we have $v\in H^{2}_{loc}$. 
  Moreover, multiplying the equation by $\overline{v}$ and taking the
  imaginary part we obtain the identity
  \begin{equation*}
    \Im
    \partial_{j}\{a_{jk}\partial^{b}_{k}v \overline{v}\} = 0
  \end{equation*}
  and integrating on $\Omega \cap\{|x|<R\}$, thanks to the Dirichlet
  boundary conditions we get, for $R$ large enough,
  \begin{equation*}
    \textstyle
    \int_{|x|=R}
    \Im(\overline{v}\widehat{x}\cdot \nabla^{b}v )dS=0.
  \end{equation*}
  This implies
  \begin{equation*}
    \textstyle
    \int_{|x|=R}(|\nabla^{b}v|^{2}+\lambda|v|^{2})dS=
    \int_{|x|=R}
    \bigl|\nabla^{b}v-i \sqrt{\lambda}\widehat{x}v\bigr|^{2}dS
  \end{equation*}
  and hence condition \eqref{eq:condinf} is satisfied. Then applying
  the previous estimate with $f=0$, $\epsilon=0$, we obtain 
  that $v \equiv0$. The last claim is proved by contradiction:
  if 
  $\int_{|x|=R}
    |\nabla^{b}v-i \sqrt{\lambda}\widehat{x}v|^{2}dS\ge \sigma$
  for some constant $\sigma>0$, then multiplying by $|x|^{\delta-1}$
  and integrating in the radial variable we obtain that the
  quantity \eqref{eq:deltawei} can not be finite.
\end{proof}

\begin{lemma}[]\label{lem:localbound}
  Assume (A), with $\mu,\lambda$ arbitrary, and let
  \begin{equation*}
    \Gamma=
    \|a-I\|_{L^{\infty}}
    +
    \||x|^{2}c_{-}\|_{L^{\infty}(|x|\le2)}.
  \end{equation*}
  Let $v\in H^{2}_{loc}(\Omega)$ with $v\vert_{\partial \Omega}=0$,
  $\lambda,\epsilon\in \mathbb{R}$ 
  and let $f=(L+\lambda+i \epsilon)v$.
  Then, if $\Gamma$ is sufficienty small with respect to $n$,
  for all $R>0$ we have
  \begin{equation}\label{eq:localbound}
    \int_{\Omega\cap\{|x|\le R\}}
    |\nabla^{b}v|^{2}
    \le
    C
    \int_{\Omega\cap\{|x|\le R+1\}}|v|^{2}
    +
    \int_{\Omega\cap\{|x|\le R+1\}}|f|^{2}
  \end{equation}
  where $C=c(n)(1+\lambda_{+}+\|c_{-}\|_{L^{\infty}(|x|\ge1)})$.
\end{lemma}

\begin{proof}
  For any real valued test function $\psi$ we can write
  \begin{equation*}
    (L+\lambda+ i\varepsilon)(\psi v)
    =
    \psi f+(A\psi) v+2a(\nabla^bv,\nabla\psi)
  \end{equation*}
  and multiplying by 
  $\psi \overline{v}$ and rearranging the terms we get
  \begin{equation*}
  \begin{split}
    \partial_{j}\{\psi \overline{v} a_{j k}\partial_{k}^{b}(\psi v)\}
    -
    a(\nabla^{b}(\psi v),\nabla^{b}(\psi v))
    +
    &
    (\lambda+i \epsilon-c)|\psi v|^{2}
    =
    \\
    =
    &
    f\psi^{2}\overline{v}
    +
    (A \psi)\psi|v|^{2}
    +2a(\nabla^{b}v,\nabla \psi)\psi \overline{v}.
  \end{split}
  \end{equation*}
  Now we take the real part and use the fact that
  \begin{equation*}
  \begin{split}
    2\Re a(\nabla^{b}v,\nabla \psi)\psi \overline{v}
    =
    2\Re a(\nabla v,\nabla \psi)\psi \overline{v}
    =&
    \textstyle
    -\frac12 a(\nabla|\psi|^{2},\nabla|v|^{2})
    \\
    =&
    \textstyle
    -\frac12(A|\psi|^{2})|v|^{2}
    -
    \partial_{j}\{\frac12a_{j k}|v|^{2} \partial_{k}|\psi|^{2}\}
  \end{split}
  \end{equation*}
  and we obtain
  \begin{equation*}
  \begin{split}
    \textstyle
    \partial_{j}\{\Re\psi \overline{v} a_{j k}\partial_{k}^{b}(\psi v)
        +\frac12a_{j k}|v|^{2} \partial_{k}|\psi|^{2}\}
    =
    a(\nabla^{b} & (\psi v),\nabla^{b}(\psi v))
    +
    (c-\lambda)|\psi v|^{2}
    +
    \\
    +
    &
    \textstyle
    \Re f\psi^{2}\overline{v}
    +
    (A \psi)\psi|v|^{2}
    -\frac12(A|\psi|^{2})|v|^{2}.
  \end{split}
  \end{equation*}
  Integrating on $\Omega$ and using
  $A|\psi|^{2}=2 \psi A \psi+2a(\nabla \psi,\nabla \psi)$
  and the Dirichlet boundary conditions, we arrive at
  \begin{equation}\label{eq:lastid}
    \textstyle
    \int_{\Omega}
    a(\nabla^{b} (\psi v),\nabla^{b}(\psi v))
    =
    \int_{\Omega}
    (\lambda-c)|\psi v|^{2}
    -
    \int_{\Omega}
    \Re f \psi^{2}\overline{v}
    +
    \int_{\Omega}
    a(\nabla \psi, \nabla \psi)|v|^{2}.
  \end{equation}
  It is clear that this identity holds for any compactly supported,
  piecewise $C^{1}$ weight function $\psi$.

  We introduce now a cutoff function $\chi$ equal to 1 in $|x|\le1$,
  equal to 0 for $|x|\ge2$, and such that $0\le \chi\le 1$.
  Then we can write
  \begin{equation*}
    \textstyle
    -\int c|\psi v|^{2}
    \le
    \int (1-\chi) c_{-}|\psi v|^{2}
    +
    \int \chi c_{-}|\psi v|^{2}.
  \end{equation*}
  We estimate the first term simply as follows:
  \begin{equation*}
    \textstyle
    \int_{\Omega} (1-\chi) c_{-}|\psi v|^{2}
    \le
    \|c_{-}\|_{L^{\infty}(|x|\ge1)}
    \int_{\Omega}|\psi v|^{2}.
  \end{equation*}
  On the other hand, for the second term we use the magnetic
  Hardy inequality:
  \begin{equation*}
    \textstyle
    \int_{\Omega} \chi c_{-}|\psi v|^{2}
    \le
    \||x|^{2}c_{-}\|_{L^{\infty}(|x|\le2)}
    \int_{\Omega}|x|^{-2}|\psi v|^{2}
    \le
    c(n)\Gamma \int_{\Omega}|\nabla^{b}(\psi v)|^{2}.
  \end{equation*}
  Since $a\ge(1-\Gamma)I$, if $\Gamma$ is sufficiently small
  with respect to $n$ we can absorb the last term at the left hand
  side of \eqref{eq:lastid} and we obtain the estimate
  \begin{equation*}
    \textstyle
    \int_{\Omega}|\nabla^{b}(\psi v)|^{2}
    \le
    c(n)(1+\lambda_{+}+\|c_{-}\|_{L^{\infty}(|x|\ge1)})
    \int_{\Omega}|\psi v|^{2}+
    \int_{\Omega}a(\nabla \psi,\nabla \psi)|v|^{2}+
    \int_{\Omega}|\psi f|^{2}.
  \end{equation*}
  Finally, we choose $\psi$ as follows: for a given $R>0$,
  \begin{equation*}
    \psi=1
    \quad\text{if}\quad |x|\le R,
    \qquad
    \psi=0
    \quad\text{if}\quad |x|\ge R+1,
    \qquad
    \psi=R+1-|x|
    \quad\text{elsewhere.}
  \end{equation*}
  Plugging $\psi$ in the previous estimate we obtain the claim.
\end{proof}

We are ready to conclude the proof of Theorem \ref{the:existence}.
Given $f$ with $\int|x|^{\delta}\bra{x}|f|^{2}<\infty$,
we consider a sequence
$\epsilon_{k}>0$ with $\epsilon_{k}\to0$ and define 
$v_{k}$ as the
unique solution $v_{k}\in H^{1}_{0}(\Omega)\cap H^{2}(\Omega)$ of
\begin{equation*}
  (L+\lambda+i \epsilon_{k})v_{k}=f.
\end{equation*}
We now remark that under the assumptions of
Theorem \ref{the:existence}, if $\kappa$ is sufficiently small,
all the conditions in
both Theorems \ref{the:smoo}
and \ref{the:sommlarge}
are satisfied.
Then, introducing the norm
\begin{equation*}
  \textstyle
  \|w\|_{\dot Z}
  :=
  \|w\|_{\dot X}+|\lambda|\|w\|_{\dot Y}+
  \|\nabla^{b}w\|_{\dot Y}
  +(n-3)\left\|\frac{|w|^{2}}{|x|^{3/2}}\right\|
  +
  (\int|x|^{\delta-1}|\nabla^{b}_{S}v|^{2}dx)^{1/2},
\end{equation*}
we get the bound
(uniform in $|\epsilon|<\lambda$ for fixed $\lambda$)
\begin{equation}\label{eq:unifbd}
  \textstyle
  \|v_{k}\|_{\dot Z}^{2}
  \lesssim
  \int|x|^{\delta}\bra{x}|f|^{2}
\end{equation}
since the last norm controls $\|f\|_{\dot Y^{*}}$.
Note on the other hand that the smoothing estimate
\begin{equation}\label{eq:unifsmoo}
  \textstyle
  \|v_{k}\|_{\dot X}
  +
  |\lambda|^{1/2}\|v_{k}\|_{\dot Y}
  +
  \|\nabla^{b}v_{k}\|_{\dot Y}+
  \|(a \nabla^{b}v_{k})_{T}\|_{L^{2}}+
  (n-3)
  \left\|\frac{v_{k}}{|x|^{3/2}}\right\|_{L^{2}}
  \le
  c(n)
  \|f\|_{\dot Y^{*}}
\end{equation}
is uniform for all $\lambda> \overline{\sigma}\cdot(K+K^{2})$ 
and all $\epsilon$.

From \eqref{eq:unifbd} we deduce that $v_{k}$
is a bounded sequence in $H^{1}(\Omega\cap\{|x|<R\})$ for all $R>0$;
by a diagonal procedure and the compact embedding of 
$H^{1}$ into $L^{2}$
we can extract a subsequence, which we denote
again by $v_{k}$, strongly convergent in 
$L^{2}(\Omega\cap\{|x|<R\})$ for all $R>0$.
Moreover, the difference $v_{k}-v_{h}$ of two solutions satisfies
the equation
\begin{equation*}
  (L+\lambda+i \epsilon_{k})(v_{k}-v_{h})=
  (\epsilon_{k}-\epsilon_{h})v_{h},
\end{equation*}
hence by Lemma \ref{lem:localbound} we see that $v_{k}$ is
a Cauchy sequence in $H^{1}(\Omega\cap\{|x|<R\})$,
and in conclusion $v_{k}$ converges strongly in
$H^{1}(\Omega\cap\{|x|<R\})$ for all $R>0$ to a limit $v$.
Clearly $v\in H^{1}_{loc}(\Omega)$, $v\vert_{\partial\Omega}=0$, and $v$
is a solution of
\begin{equation*}
  (L+\lambda)v=f.
\end{equation*}
We note that by \eqref{eq:unifbd} the sequence $v_{k}$ is bounded
in $\dot Z$ which is the dual of a separable space, hence
it admits a weakly-* convergent subsequence whose limit
satisfies the same bound. This means that $v\in \dot Z$ with
\begin{equation*}
  \textstyle
  \|v\|_{\dot Z}^{2}
  \lesssim
  \int|x|^{\delta}\bra{x}|f|^{2},
\end{equation*}
and that $v$ satisfies also the smoothing estimate
\eqref{eq:unifsmoo}.

Finally, if we apply the same procedure to any
subsequence of the original sequence,
we can extract from it a subsequence which converges in
$H^{1}_{loc}$ strongly and in $\dot Z$ weakly-* to a solution
$\widetilde{v}$ of the Helmholtz equation satisfying the same bounds,
and by Corollary \ref{cor:uniq} we must have $\widetilde{v}=v$.
This implies that the entire original sequence converges to
$v$ both in 
$H^{1}_{loc}$ strongly and in $\dot Z$ weakly-*,
and the proof is concluded.

\section{Acknowledgments}
The third author is supported by the ERC Starting Grant 676675 FLIRT.


\end{document}